\documentclass[12pt]{article}

\usepackage[a4paper,includeheadfoot,margin=2.54cm]{geometry} 
\usepackage[utf8]{inputenc}
\usepackage{cite} 
\usepackage{authblk}
\usepackage{amsmath,amsfonts,amssymb}
\usepackage{bm}
\usepackage{color}
\usepackage[english]{babel}
\usepackage{graphicx}%
\usepackage{todonotes}
\usepackage{epstopdf}
\usepackage{subcaption}
\usepackage{booktabs}
\usepackage{color,soul}
\usepackage[margin=10pt,font=small,labelfont=bf]{caption}
\usepackage{stmaryrd}
\usepackage{mathtools}

\usepackage{hyperref}

\numberwithin{equation}{section}

\newcommand{\llb}{\llbracket}
\newcommand{\rrb}{\rrbracket}

\usepackage{mathabx}
\newcommand{\tn}{\vvvert}

\newcommand{\lla}{\{}
\newcommand{\rra}{\}}

\newcommand{\mcK}{\mathcal{K}}
\newcommand{\mcL}{\mathcal{L}}
\newcommand{\mcA}{\mathcal{A}}
\newcommand{\mcB}{\mathcal{B}}
\newcommand{\mcM}{\mathcal{M}}
\newcommand{\mcN}{\mathcal{N}}
\newcommand{\mcF}{\mathcal{F}}
\newcommand{\mcO}{\mathcal{O}}
\newcommand{\mcG}{\mathcal{G}}
\newcommand{\mcI}{\mathcal{I}}

\newcommand{\hatu}{\widehat{u}}
\newcommand{\hatv}{\widehat{v}}
\newcommand{\hatw}{\widehat{w}}
\newcommand{\hatg}{{\widehat{g}}}
\newcommand{\hata}{\widehat{a}}
\newcommand{\hatb}{\widehat{b}}
\newcommand{\hatn}{\widehat{n}}

\newcommand{\hatD}{\widehat{D}}
\newcommand{\hatx}{{\widehat{x}}}
\newcommand{\hatf}{\widehat{f}}
\newcommand{\hate}{\widehat{e}}
\newcommand{\hatgamma}{\widehat{\gamma}}
\newcommand{\hatomega}{\widehat{\omega}}
\newcommand{\hatnu}{\widehat{\nu}}
\newcommand{\hattau}{\widehat{\tau}}
\newcommand{\hatK}{{\widehat{K}}}
\newcommand{\hatV}{\widehat{V}}
\newcommand{\hatF}{\widehat{F}}
\newcommand{\hatG}{\widehat{G}}

\newcommand{\tildeg}{\widetilde{g}}

\newcommand{\hatnabla}{\widehat{\nabla}}
\newcommand{\hatpartial}{\widehat{\partial}}
\newcommand{\hatOmega}{\widehat{\Omega}}
\newcommand{\hatGamma}{\widehat{\Gamma}}
\newcommand{\hatmcK}{\widehat{\mcK}}
\newcommand{\ia}{i}
\newcommand{\ib}{j}
\newcommand{\Ii}{\mcI_\Omega}
\newcommand{\Iij}{\mcI_\Gamma}

\newcommand{\IR}{\mathbb{R}}
\newtheorem{lem}{Lemma}[section]
\newtheorem{ass}{Assumption}[section]

\newtheorem{thm}{Theorem}[section]
\newtheorem{rem}{Remark}[section]

\newenvironment{proof}{\noindent \newline {\bf Proof.}}
{\hfill \mbox{\fbox{} } \newline}

\newcommand{\divv}{\text{div}}

\newcommand{\map}{F}
\newcommand{\domain}{\Omega}
\newcommand{\refdomain}{\widehat{\Omega}}

\newcommand{\interface}{\Gamma}

\newcommand{\interp}{\pi}
\DeclareMathOperator{\Span}{span}

\newcommand{\glue}{\beta}
\newcommand{\stab}{\gamma}
\newcommand{\Patches}{\mcO}
\newcommand{\Interfaces}{\mcG}

\newcommand{\Fhi}{{\widehat{\mathcal{F}}_{h,i}}}

\title{\bf Cut Finite Element Methods for\\ Elliptic Problems on Multipatch\\ Parametric Surfaces\thanks{This research was supported in part by the Swedish Foundation for Strategic 
Research Grant No.\ AM13-0029, the Swedish Research Council Grant No.\ 2013-4708, and the Swedish strategic research programme eSSENCE}}
\date{}
\author{Tobias Jonsson\footnote{tobias.jonsson@umu.se}}
\author{Mats G. Larson\footnote{mats.larson@umu.se}}
\author{Karl Larsson\footnote{karl.larsson@umu.se}}
\affil{\it Department of Mathematics and Mathematical Statistics, Ume{\aa}~University, SE-901\,87 Ume{\aa}, Sweden}

\begin{document}

\maketitle

\begin{abstract}
We develop a finite element method for the Laplace--Beltrami operator  
on a surface described by a set of patchwise parametrizations. The 
patches provide a partition of the surface and each patch is the image 
by a diffeomorphism  of a subdomain of the unit square which is 
bounded by a number of smooth trim curves. A patchwise tensor product 
mesh is constructed by using a structured mesh in the reference domain. 
Since the patches are trimmed we obtain cut elements in the vicinity of 
the interfaces. We discretize the Laplace--Beltrami operator using a cut 
finite element method that utilizes Nitsche's method to enforce continuity 
at the interfaces and a consistent stabilization term to handle the 
cut elements. Several quantities in the method are conveniently 
computed in the reference domain where the mappings impose a Riemannian 
metric. We derive a priori estimates in the energy and $L^2$ norm and also 
present several numerical examples confirming our theoretical results.
\end{abstract}

\paragraph{Subject Classification Codes:} 65M60, 65M85.

\clearpage
\tableofcontents
\clearpage

\section{Introduction}
\paragraph{Background.}
Differential equations on surfaces appear in many applications including transport phenomena 
on surfaces and elastic membranes and shells. In engineering applications the surface geometry 
is often described using a CAD model consisting of a partition of the surface into trimmed patches 
defined by mappings from a reference domain onto the surface. When performing computations on surfaces it is beneficial to directly utilize the available parametric geometry description, in line with 
the ideas of isogeometric analysis (IGA) \cite{IGA}.

There are various techniques of enforcing interface conditions between patches, for example 
Lagrange penalty methods or methods based on weak enforcement. In the case of thin shells another approach is the bending strip method \cite{KiBaHsWuBl10}. The use of Nitsche's method \cite{Nitsche71}, or variants thereof, to weakly enforce interface conditions between patches is a well established and flexible technique, see for example \cite{ApoSch14, NitscheIGA,Langer2015,Guo2017} and the references therein. However, constructing a high quality conforming mesh on the trimmed patches is generally a difficult task. 
In this work we address the problem of conforming mesh construction by allowing the trim curve 
on each patch to arbitrarily cut the mesh by utilizing a fictitious domain method called the cut finite 
element method (CutFEM) \cite{BH12,BCHLM15}. In the same spirit \cite{RueSch13,RueSch14, KolOzc15} allow cut elements but employ the finite cell method, which is based on a different 
stabilization mechanism, where a small artificial stiffness is added on the part of the cut element 
outside of the patch it belongs to.

\paragraph{Contributions.}
We develop a general technique for consistent discretization and a framework for analysis of the 
Laplace--Beltrami operator, which serves as a model second order partial differential operator 
on a patchwise parametric surface. The patches provide a partition of the surface and each 
patch is assumed to be the image by a diffeomorphism  of a subdomain of the unit square which is bounded by a number of smooth trim curves. A patchwise tensor product mesh is constructed by 
using a structured mesh in the reference domain. Since the patches are trimmed we obtain cut 
elements in the vicinity of the interfaces. We discretize the Laplace--Beltrami operator using 
a cut finite element method that utilizes Nitsche's method to enforce continuity at the interfaces 
and a consistent stabilization term to handle the cut elements. 

Several quantities in the method are conveniently computed in the reference domain where 
the mappings impose a Riemannian metric. In particular, the stabilization term only involves 
derivatives in the reference coordinates, which is convenient since it involves higher order 
derivatives. We develop a quadrature formula for integration on the cut elements that is applicable 
to a piecewise smooth boundary. We show that the method is stable and we derive optimal order 
a priori estimates in the energy and $L^2$ norms and also present several numerical examples 
confirming our theoretical results.

Summarizing the key characteristics of the technique and framework for analysis developed in this work are:
\begin{itemize}
\item The method for computations on multipatch parametric surfaces is based on a fictitious domain method (CutFEM) which does not require the construction of conforming meshes for the trimmed patches.
\item	A stabilization term is added which allows us to perform a complete stability and error analysis independent of how the trim curves cut the computational mesh. In particular, we include proofs of 
the basic estimates related to the stabilization term in the case of higher order parametric polynomial spaces, as well as an estimate of the condition number of the stiffness matrix.

\item Both the method and the analysis are adapted to higher order elements and a quadrature rule for integration of cut higher order elements is suggested. We consider standard Lagrange elements but 
the analysis may also be applied to the spline spaces used in isogeometric analysis.
\end{itemize}

\paragraph{Outline.} In Section~\ref{sec:the-surface} we define the patchwise parametric surface, 
recall some basic facts on differential operators on surfaces, and formulate our 
model problem, in Section~\ref{sec:fem} we construct the patchwise mesh, the finite element 
spaces, formulate the finite element method, and provide some details on the implementation 
including a method for quadrature on cut elements,  
in Section~\ref{sec:error-est}  we prove stability of the method, construct an interpolation operator, 
prove a priori error estimates in the energy and $L^2$ norm and prove an upper bound for the stiffness matrix condition number. In 
Section~\ref{sec:numerics} we present numerical results confirming our theoretical results. Finally, in Section~\ref{sec:conclusions} we summarize our results and comment on possible future developments.

\section{The Surface and the Laplace--Beltrami Operator} \label{sec:the-surface}

\subsection{Piecewise Parametric Description of the Surface}
\label{sec:defsurface}
We define the surface and the piecewise parametrization as follows:
\begin{itemize}
\item Let $\domain$ be a piecewise smooth connected surface immersed in $\mathbb{R}^d,  d\geq 2$, which is not necessarily orientable.

\item For all points $x\in\overline{\Omega}$ let $\gamma_x(s)= \{y\in\Omega \, : \, \mathrm{dist}(x,y)=s\}$ define a path on $\Omega$ at a fixed distance $s$ to $x$. The length of this path is denoted $|\gamma_x(s)|$.

\item If $\Omega$ has a boundary $\partial\Omega$ it is assumed to be described by a set of smooth curves. Furthermore, for all points $x\in\partial\Omega$ we require $\lim_{s\rightarrow 0+}\frac{|\gamma_x(s)|}{s}=\pi$ which means that while the boundary may include kinks, in an intrinsic sense it is smooth.

\item For all points $x\in\Omega$ we require $\lim_{s\rightarrow 0+}\frac{|\gamma_x(s)|}{s}=2\pi$ which means that while the surface may include sharp edges, in an intrinsic sense it is smooth. More concretely the surface can have sharp edges, like the surface of a cylinder, but is not allowed to have corners, like the surface of a cube.

\item Let $\mcO = \{\Omega_\ia, \ia \in \Ii\}$ be a partition of 
$\Omega$ into a finite number of smooth subdomains $\Omega_\ia$ which we denote patches.

\item We assume that each patch
boundary $\Gamma_\ia = \partial \Omega_\ia$
is described by a uniformly bounded number of smooth curves.

\item The interfaces between the patches in $\mcO$ are described by the curves in the set $\mcG=\{ \Gamma_{\ia\ib}=\overline{\Omega}_\ia \cap \overline{\Omega}_\ib, (i,j)\in\mcI_\Gamma, i<j\}$ where $\mcI_\Gamma$ is a set of pairs of domain indices $i,j\in\mcI_\Omega$ for neighboring patches.

\item The boundary is described by the curves in the sets
$\mcB_{\partial\Omega_D}$ and $\mcB_{\partial\Omega_N}$ for the Dirichlet and Neumann parts of the boundary, respectively.

\item For each patch $\Omega_\ia$ we associate a diffeomorphism 
$F_\ia^{-1}:\Omega_\ia\rightarrow \widehat{\Omega}_\ia \subset 
I^2=[0,1]^2$ to the reference domain. We also assume that $F_\ia|_{\hatOmega_\ia}$,
is the restriction of a diffeomorphism $F_\ia:U_\delta(\hatOmega_\ia)\rightarrow F_\ia U_\delta(\hatOmega_\ia)$, where  
$U_\delta(\hatOmega_\ia) = \{\hatx \in I^2: d(\hatx,\hatOmega_\ia)<\delta\}$ and $d(x,y) = \| x - y \|_{\IR^2}$ is the usual Euclidean distance function.

\item To be able to evaluate functions $u \in H^{p+1}(\Omega)$ on $F_\ia U_\delta(\hatOmega_\ia)$  we for patches $\hatOmega_\ia$ with a smooth interface further assume $F_\ia U_\delta(\hatOmega_\ia)
\subset \Omega$ while we for patches with a sharp interface define an $H^{p+1}$-extension $E^* u$ of $u$ on $F_\ia U_\delta(\hatOmega_\ia)$ which is possible as $\partial\Omega_i$ in the sharp interface case is a closed smooth curve.

\item Given $x\in \Omega_\ia$ we denote the corresponding 
point $F_\ia^{-1} x \in {\hatOmega}_\ia$ by $\hatx$ and 
given a subset $\omega \subset \Omega_\ia$ we let 
$\hatomega = F_\ia (\omega)$. For 
a function $v:F_\ia U_\delta(\hatOmega_\ia) \rightarrow \IR$ we 
let $\hatv:U_\delta(\hatOmega_\ia) \rightarrow \IR$ 
denote the pullback $\hatv = v \circ F_\ia$ and for a function 
$\hatv:U_\delta(\hatOmega_\ia) \rightarrow \IR$ we let 
$v:F_\ia U_\delta(\hatOmega_\ia) \rightarrow \IR$ denote the push 
forward such that $v\circ F_\ia=\hatv$. Note that the pullback is 
indeed defined on the slightly larger domain $U_\delta(\hatOmega_\ia)$. 
This property will be convenient when we construct an interpolation operator. 
\end{itemize}

This surface description and notation are illustrated in Figure~\ref{fig:patch-mapping}.

\begin{figure}
\centering
\includegraphics[width=0.8\linewidth]{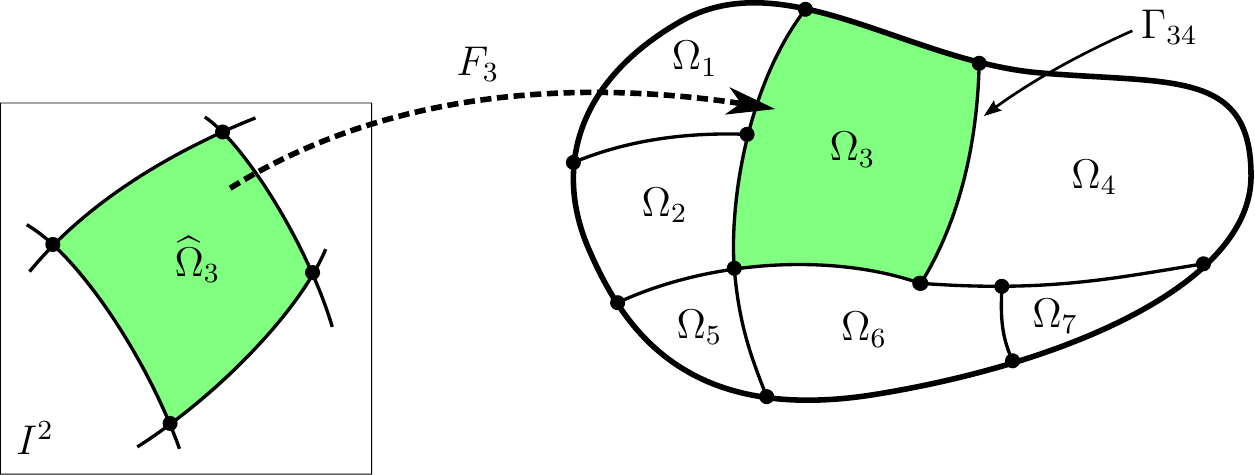}
\caption{Representation of the surface via patchwise parametrizations.}
\label{fig:patch-mapping}
\end{figure}

\subsection{Riemannian Metric on the Reference Patches}
\paragraph{Tangent Spaces.}
At each point $\hatx \in \hatOmega_\ia$ we let 
$T_{\hatx}(\hatOmega_\ia)$ denote the tangent space of 
$\hatOmega_\ia$ and we let $\{\hate_1,\hate_2\}$ be a fixed orthonormal 
basis for $T_{\hatx}(\hatOmega_\ia)$, i.e. the same basis is used 
independent of $\hatx$. At each point $x \in \domain_{\ia}$ we locally
define the tangent space to $\Omega_\ia$ as
\begin{align}
T_{x}(\domain_\ia) 
= \Span\{\hatpartial_1 F_{\ia},\hatpartial_2 F_{\ia}\}
\end{align}
where $\hatpartial_k$ is the partial derivative in the 
$\hate_k$ direction and $\hatpartial_k F_{\ia} 
= \hatpartial_k F_{\ia}|_{x}$ . Thus, any tangent vector $a \in T_{x}(\domain_i)$ 
can be written
\begin{align}
a = \hata_1 \hatpartial_1 F_{\ia} + \hata_2 \hatpartial_2 F_\ia 
= \begin{bmatrix} \hatpartial_1 F_\ia &  \hatpartial_2 F_\ia 
\end{bmatrix}
\begin{bmatrix}
\hata_1 \\ \hata_2
\end{bmatrix}
= \hatD F_\ia \hata
\label{eq:goidjn}
\end{align}
where $\hata=[\hata_1,\hata_2]^T \in T_{\hatx}(\hatOmega_\ia)$ and we introduced the 
notation 
\begin{equation}
\hatD F_\ia = F_\ia \otimes \hatnabla = \begin{bmatrix} \hatpartial_1 F_\ia &  \hatpartial_2 F_\ia 
\end{bmatrix}
\end{equation}

\paragraph{Riemannian Metric.}
We equip $T_x(\domain_\ia)$ with the Euclidean inner product "$\cdot$" in $\mathbb{R}^d$, i.e. the inner product of the immersing space, and we define the induced inner product $\hatg_{\hatx}$ on $T_{\hatx}(\refdomain_{\ia})$ as follows
\begin{align}
\hatg_{\hatx}(\hata,\hatb) = a \cdot b
\end{align}
Introducing the symmetric positive definite matrix $\hatG$ with components $\hatg_{kl} = \hatpartial_k F_\ia \cdot \hatpartial_l F_\ia$  the inner 
product $\hatg_{\hatx}(\hata,\hatb)$ can be written
\begin{align}
\hatg_{\hatx}(\hata,\hatb) =
\begin{bmatrix}
\hata_1 & \hata_2
\end{bmatrix}
\begin{bmatrix} \hatg_{11} & \hatg_{12}\\ \hatg_{21} & \hatg_{22} \end{bmatrix} 
\begin{bmatrix}
\hatb_1 \\ \hatb_2
\end{bmatrix}
= \hata^T \hatG \hatb
\end{align}
This inner product is a Riemannian metric on $\hatOmega_\ia$ 
and $\hatG$ is called the metric tensor. We denote the norm on 
$T_{\hatx}(\hatOmega_\ia)$ induced by the Riemannian metric by
\begin{align}
\| \hatv \|_{\hatg_{\hatx}}^2 = \hatg_{\hatx}(\hatv,\hatv)
\end{align}
We note that $\hatD F_i :T_{\hatx}(\hatOmega_\ia) \ni \hata 
\mapsto a \in
T_{x}(\Omega_\ia)$ is an isometry since norms and angles are 
preserved
\begin{align}
\| a \|_{\mathbb{R}^d}^2=
a \cdot a = \hatg_{\hatx}(\hata,\hata)=\| \hata \|_{\hatg_{\hatx}}^2
\end{align}
and if $\theta$ ($\widehat{\theta}$) is the angle between $a$ and $b$ 
($\hata$ and $\hatb$) we have 
 \begin{align}
\cos \theta  = \frac{a \cdot b}{\|a\|_{\mathbb{R}^d} \| b \|_{\mathbb{R}^d}} 
&= \frac{\hatg(\hata,\hatb)}{ \| \hata \|_{\hatg_\hatx} \| \hatb \|_{\hatg_\hatx}} 
= \cos \widehat{\theta}
\end{align}
We note that given $a \in T_{x}(\Omega_i)$ we find the corresponding $\hata \in T_\hatx(\hatOmega_i)$ using  the 
relation 
\begin{equation}
\hata = \hatG^{-1} \hatD F_\ia^T a
\end{equation}
since we have the identity  
\begin{equation} 
\hatg_{\hatx}(\hata,\hatb) 
= a \cdot b 
= a \cdot (\hatD F_\ia \hatb) 
=(\hatD F_\ia^T a) \cdot \hatb\qquad \forall \hatb \in T_{\hatx}(\hatOmega_\ia) 
\end{equation}
and thus we conclude that $\hatG \hata = \hatD F_\ia^T a$.  

Note that we have a uniform bound on the eigenvalues of $\hatG$, i.e. there exists constants $0 < c < C$ such that
\begin{equation}
c  \leq \lambda_\mathrm{min}( \hatG )
\leq \lambda_\mathrm{max}( \hatG ) \leq C
\qquad \forall \hatx \in \hatOmega_\ia
\end{equation}
and as a consequence
\begin{equation}
c \| \hata \|^2_{\hatg_{\hatx}} \leq \| \hata \|_{\IR^2}^2 \leq C \| \hata \|^2_{\hatg_{\hatx}}
\qquad \hatx \in \hatOmega_\ia, \ \hata \in T_{\hatx}(\hatOmega_\ia)  
\end{equation}

We will use the notation $T(\omega) = \sqcup_{x \in \omega} T_x(\omega)$  for the tangent bundle over a subset $\omega\subset\Omega$, which is the collection of all the 
tangent vector spaces at the points in $\omega$.
When there is no possibility of confusion we will use the simplified notation $g=g_x$, $\hatg=\hatg_\hatx$ and correspondingly for the induced norms.

\subsection[Integration and $L^2$ Inner Products]{Integration and \texorpdfstring{${\boldsymbol L}^{\boldsymbol 2}$}{L2} Inner Products} \label{section:L2innerprod}
\paragraph{Integration.}
The integral over $\omega\subset \Omega_\ia$ is defined  by
\begin{align}
\int_\omega f dx &= \int_{\hatomega} \hatf \ |\hatG|^{1/2} \, d\widehat{x}
\end{align}
where $|\hatG| = |\text{det}(\hatG)|$.  

Let $\hatgamma$ be a curve in $U_\delta(\hatOmega_\ia)$ 
and let $(0,l) \ni s \mapsto \hatgamma(s) \in U_\delta(\hatOmega_\ia)$ 
be an arclength parametrization and  $d\hatgamma$ the arclength 
measure. We define the integral over the curve 
$\gamma = F_\ia \circ \hatgamma \subset \Omega_\ia$ as 
follows
\begin{equation}
\int_\gamma f d\gamma 
= \int_{\hatgamma} \hatf \|\widehat{\tau}\|_{\hatg} d\hatgamma
= \int_0^l \hatf \circ \hatgamma(s) \|\widehat{\tau}\circ \hat{\gamma}(s)\|_{\hatg_{\hatgamma(s)} } ds
\end{equation}
where $\widehat{\tau} = \frac{d\hatgamma}{ds}$ is the unit tangent vector to $\hatgamma$ with 
respect to the Euclidean $\IR^2$ inner product.

\begin{rem} In Appendix \ref{appendix:operators} we provide some more details on the definition 
of the integrals and also discuss the extension to higher dimensions.
\end{rem}

\paragraph{Inner Products.}
We let $(\cdot,\cdot)_{\omega}$ denote the $L^2(\omega)$ inner product 
\begin{align}
(v,w)_{{\omega}} &=
\left\{
\begin{alignedat}{2}
&\int_{{\omega}} v w  \, dx 
= 
\int_{{\widehat{\omega}}} \hatv \hatw  \, |\hatG|^{1/2}d\hatx
\qquad &&\text{for $v,w: {\omega} \rightarrow \mathbb{R}$}
\\
&\int_{{\omega}} v \cdot w \, dx 
=\int_{{\widehat{\omega}}} \hatg(\hatv,\hatw)  \, |\hatG|^{1/2} d\hatx
\qquad &&\text{for $v,w: {\omega} \rightarrow T(\omega)$}
\end{alignedat}
\right.
\label{eq:L2prod}
\end{align}
and for curves $\gamma$ we have analogous definitions where the integrals 
and measures are replaced by integrals over the curve and the appropriate measures.

\subsection{Differential Operators in Reference Coordinates} \label{section:diffoper}
Here we introduce some differential operators and formulate 
Green's formula. In Appendix \ref{appendix:operators} we also derive these 
expressions including Green's formula using basic calculus.

\paragraph{The Gradient.}
Let $\hatnabla$ be the tangential gradient operator in reference 
coordinates
\begin{equation}
\widehat{\nabla}=\begin{bmatrix} \hatpartial_1 \\ \hatpartial_2 \end{bmatrix}
\end{equation}
The tangential gradient $\nabla u \in T_p(\Omega_i)$ is 
represented in terms of reference coordinates 
\begin{equation}
\nabla u = \hatD F_\ia \widehat{\nabla u}
\end{equation}
Using the chain rule we obtain the identities
\begin{align}
\hatnabla \hatu 
= \hatnabla (u \circ F_\ia) 
&=
(\hatD F_\ia)^T  \nabla u
=
(\hatD F_\ia)^T  \hatD F_\ia \widehat{\nabla u} 
= \hatG \widehat{\nabla u} 
\end{align}
and thus we conclude that the the gradient representation in reference coordinates is
\begin{align}\label{eq:gradient-local}
\widehat{\nabla u}
&= \hatG^{-1} \widehat{\nabla} \hatu
\end{align}

\paragraph{The Divergence.} The divergence operator 
$\divv$ on $\Omega$ is defined by the identity
\begin{equation}\label{def:divergence}
(\divv v, w)_\Omega = -(v,\nabla w)_\Omega
\end{equation}
for $w \in C^\infty_0(\omega)$,
and may be expressed in reference coordinates as follows
\begin{equation}
\widehat{\divv v} = |\hatG|^{-1/2} \hatnabla \cdot \Big( |\hatG|^{1/2} \hatv \Big)
\end{equation}

\paragraph{The Laplace--Beltrami Operator.} We define 
the Laplace--Beltrami operator $\Delta$ on the surface $\Omega$ by
\begin{equation}\label{def:Laplace-Beltrami}
\Delta v = \divv (\nabla v)%
\end{equation}
which in reference coordinates is given by the identity 
\begin{equation}\label{def:Laplace-Beltrami-local}
\widehat{\Delta v} = |\hatG|^{-1/2} \hatnabla \cdot (|\hatG|^{1/2} \hatG^{-1} \hatnabla \hatv )
\end{equation}

\paragraph{Green's Formula.} Green's formula on $\omega \subset \Omega_\ia$ takes the form
\begin{align}\label{eq:Greens-formula}
-(\Delta v,w)_\omega = (\nabla v,\nabla w)_\omega - (\nabla v, n w)_{\partial \omega}
\end{align}
where $n \in T_p(\Omega_\ia)$ is the exterior unit normal to the curve 
$\partial \omega$.

\subsection{Sobolev Spaces}

In the reference coordinates we let $H^k(\hatOmega_\ia)$ denote 
the usual Sobolev spaces of order $k$ with inner product and norm
\begin{equation}
(\hatv,\hatw)_{H^k(\hatOmega_\ia)} = \sum_{|a|\leq k} (\hatD^a \hatv,\hatD^a \hatw)_{L^2(\hatOmega_\ia)},\qquad 
\|\hatv\|^2_{H^s(\hatOmega_\ia)}=(\hatv,\hatv)_{H^k(\hatOmega_\ia)}
\end{equation}
On the surface we define the corresponding spaces of functions $v$ that 
are liftings of functions $\hatv \in H^k(\hatOmega_\ia)$, 
\begin{equation}
H^k(\Omega_\ia) 
= H^k(\hatOmega_\ia)  \circ F_\ia^{-1}
= \{v: \Omega \rightarrow \IR \,|\, v = \hatv \circ F_\ia^{-1} ,  
\hatv \in H^k(\hatOmega_\ia) \}  
\end{equation}
with inner product and norm
\begin{equation}
(v,w)_{H^k(\Omega_\ia)} = (\hatv,\hatw)_{H^k(\hatOmega_\ia)},
\qquad 
\| v \|_{H^k(\Omega_\ia)}  = \| \hatv \|_{H^k(\hatOmega_\ia)}  
\end{equation}
We employ standard notation $L^2(\omega)=H^0(\omega)$ and $\| \cdot \|_{L^2(\omega)}=\| \cdot \|_\omega$.

\subsection{The Laplace--Beltrami Interface Problem}

\begin{figure}
\centering
\begin{subfigure}[b]{0.47\textwidth}\centering
\qquad
\includegraphics[width=0.8\linewidth]{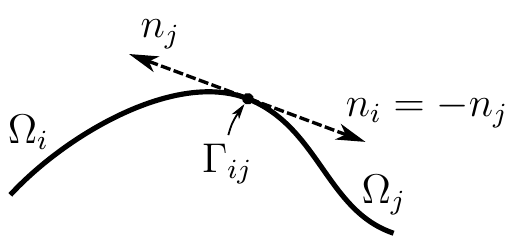}
\caption{Smooth interface}
        \label{fig:nsmooth}
\end{subfigure}
\quad
\begin{subfigure}[b]{0.47\textwidth}\centering
\includegraphics[width=0.65\linewidth]{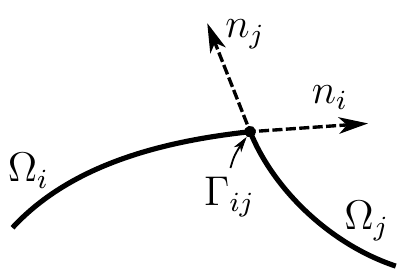}
\caption{Sharp interface}
        \label{fig:nsharp}
\end{subfigure}
\caption{Cross sections over a smooth (a) and sharp (b) patch interface $\Gamma_{ij}$ with exterior unit normals to the patch boundaries indicated.}
\label{fig:n}
\end{figure}

Let $n_i\in T_x(\Omega_i)$ be the outward pointing normal to the patch boundary $\partial \Omega_i$ as illustrated over a patch interface $\Gamma_{ij}$ in Figure~\ref{fig:n}.
We formulate our model problem for a surface without boundary:
Given $f$ such that $(f,1)_\Omega=0$, 
find $u:\Omega \rightarrow \IR$, with $(u,1)_\Omega=0$, 
such that
\begin{subequations}\label{eq:lb}
\begin{alignat}{2}
-\Delta u &= f
\qquad &&\text{in all $\Omega_\ia \in \Patches$}
\label{eq:ibdfn}
\\
\left\llb u\right\rrb  &= 0 
\qquad &&\text{on all $\Gamma_{\ia\ib} \in \Interfaces$}
\label{eq:djngj1}
\\
{n}_{\ia} \cdot \nabla u_\ia +{n}_{\ib} \cdot \nabla u_\ib  &= 0
\qquad &&\text{on all $\Gamma_{\ia\ib} \in \Interfaces$}
\label{eq:djngj2}
\end{alignat}
\end{subequations}
where the jump operator is defined by
\begin{equation}
\llb w\rrb \big|_{\interface_{\ia\ib}} = w_\ia-w_\ib
\end{equation}

\begin{rem}
We pose our model problem on a surface without boundary to simplify the presentation in the analysis. However, we comment in Section~\ref{sec:method} on how the method is easily adapted to boundary conditions and in Section~\ref{sec:numerics} we also present a numerical example with boundary conditions.
\end{rem}

\begin{rem} We use the interface formulation since we have parametric 
mappings defined on the partition $\mcO$ of $\Omega$ in contrast 
to the standard manifold description which is based on a partition of 
unity and compatibility conditions between the local parametrizations.
\end{rem}
Let $V$ be the Hilbert space
\begin{equation}\label{eq:spaceV}
V = \Bigl\{ \, v\in \bigoplus_{i\in\mcI_\Omega} H^1(\Omega_\ia) \ :\ \text{$\left\llb v \right\rrb = 0$ on $\Gamma_{ij}$, $\forall(i,j)\in\mcI_\Gamma$ \ ; \ 
$\int_\Omega v \, dx = 0$} \, \Bigr\}
\end{equation}
Then we have following weak formulation of (\ref{eq:lb}), 
\begin{align}
(f,v)_{\Omega} &= \sum_{\ia\in\Ii} (-\Delta u,v)_{\Omega_\ia} 
\\
&=  \sum_{\ia\in\Ii} (\nabla u,\nabla v)_{\Omega_\ia} - (n\cdot \nabla u, v)_{\partial \Omega_\ia} 
\\
& =  \sum_{\ia\in\Ii} (\nabla u,\nabla v)_{\Omega_\ia} 
- \ \sum_{\mathclap{(i,j)\in\Iij}}\ (n_\ia \cdot \nabla u_\ia + n_\ib \cdot \nabla u_\ib , v)_{\Gamma_{\ia\ib}} 
\\
& =  \sum_{\ia\in\Ii} (\nabla u,\nabla v)_{\Omega_\ia} 
\end{align}
for all $v \in V$. It then follows from the Lax--Milgram lemma that 
(\ref{eq:lb}) has a unique solution in $V$ for $f\in V'$, the dual of $V$, 
such that $\int_\Omega f \, dx = 0$. Furthermore, we also have the 
elliptic regularity result
\begin{equation}\label{eq:ellreg}
\sum_{\ia\in\Ii} \|v\|_{H^{s+2}(\Omega_\ia)} \lesssim \sum_{\ia\in\Ii} \|f\|_{H^{s}(\Omega_\ia)}
\end{equation}

\section{The Finite Element Method} \label{sec:fem}

\subsection{Construction of the Mesh}
Let $\widehat{\mcK}_{h,i,0}$ be a uniform 
structured tensor product mesh on the unit 
square $I^2 =[0,1]^2$ consisting of elements 
$\hatK$ with mesh size $h$. For each patch 
$\domain_\ia$ we define the active 
background mesh in the reference domain 
as
\begin{equation}
\widehat{\mcK}_{h,\ia} = \{ \hatK \in \widehat{\mcK}_{h,i,0} \, 
:\, \hatK \cap \refdomain_\ia \neq  \emptyset \}
\end{equation}
and the corresponding mesh on the surface is 
obtained by 
\begin{equation}
\mcK_{h,\ia} = \map_i( \widehat{\mcK}_{h,\ia})
= \{\map_i(\widehat{K}) : \widehat{K} \in \widehat{\mcK}_{h,\ia} \}
\end{equation}
Let $\widehat{\mcF}_{h,\ia}$ be the set of interior faces 
belonging to elements in $\widehat{\mcK}_{h,\ia}$ that 
intersects the boundary $\partial \refdomain_\ia$. The mesh and the set of edges $\widehat{\mcF}_{h,\ia}$ are illustrated in Figure~\ref{fig:meshes}. Finally, the collection of meshes 
\begin{equation}
\mcK_h = \{\mcK_{h,\ia}\, : \, \ia \in \Ii \}
\end{equation} 
provides a mesh on the surface with cut elements in the vicinity 
of the interfaces.

\begin{figure}[tb]
\centering
\begin{subfigure}[b]{0.38\textwidth}\centering
\includegraphics[width=0.85\linewidth]{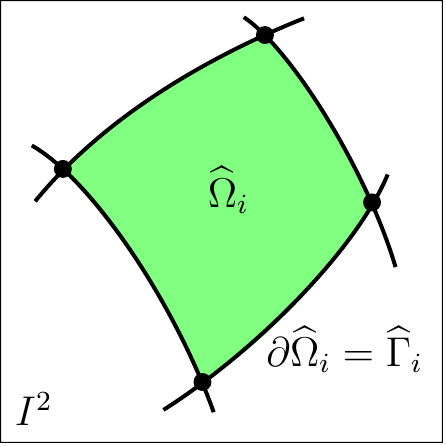}
\caption{Patch in reference domain}
        \label{fig:ref-patch}
\end{subfigure}
\begin{subfigure}[b]{0.38\textwidth}\centering
\includegraphics[width=0.85\linewidth]{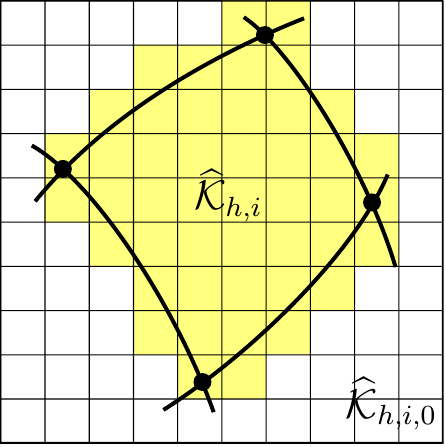}
\caption{Mesh in reference domain}
        \label{fig:ref-mesh}
\end{subfigure}

\vspace{1.0em}
\begin{subfigure}[b]{0.38\textwidth}\centering
\includegraphics[width=\linewidth]{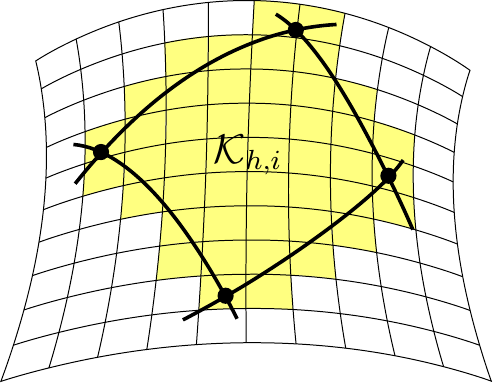}
\caption{Mesh in physical domain}
        \label{fig:phys-active}
\end{subfigure}
\begin{subfigure}[b]{0.38\textwidth}\centering
\includegraphics[width=0.85\linewidth]{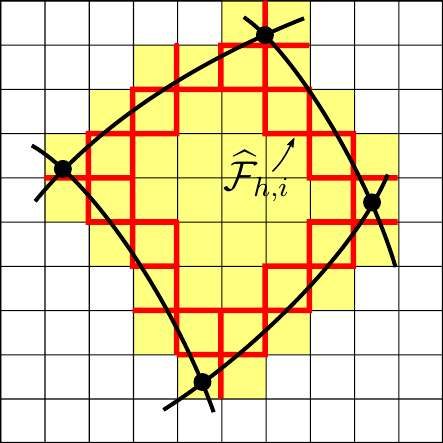}
\caption{Set of edges $\widehat{\mcF}_{h,i}$}
        \label{fig:ref-stab}
\end{subfigure}
\caption{Illustration of meshes for a patch $\Omega_i$.}
\label{fig:meshes}
\end{figure}

\subsection{The Finite Element Spaces}
Let $\widehat{V}_{h,0}$ be a finite element space on $I^2=[0,1]^2$ of continuous piecewise tensor product polynomials of order $p$ defined on the mesh $\widehat{\mcK}_{h,i,0}$. We define the spaces
\begin{align}
\widehat{V}_{h,\ia} &= \widehat{V}_{h,0}|_{\widehat{\mcK}_{h,\ia}}
\\
V_{h,i} &= \widehat{V}_{h,i} \circ F_i^{-1} = 
\{ v: \domain_i \rightarrow \mathbb{R} 
: v = \widehat{v}\circ \map_i^{-1}, \, \widehat{v} \in \widehat{V}_{h,i} \}
\\
V_h &= \bigoplus_{\ia\in \Ii} V_{h,\ia}
\end{align}
and we note that the functions in $V_h$ are discontinuous across the interfaces $\Gamma_{ij}\in\mcG$.

\begin{rem} Since each patch boundary $\widehat{\Gamma}_i = \partial\widehat{\Omega}_i$ can arbitrarily cut the mesh $\widehat{\mcK}_{h,i,0}$ we typically get cut elements on both sides of an interface and on the boundary.
\end{rem}

\subsection{The Method} \label{sec:method}

Using the formulation in \cite{flera-ytor} we introduce the Nitsche bilinear form $a_h$ given by
\begin{align}
a_h(v,w) &= \sum_{\ia\in\Ii} (\nabla v, \nabla w)_{\Omega_\ia}
\\ \nonumber
&\qquad -\sum_{(i,j)\in\Iij} 
(\{ n\cdot \nabla v \},\llb w\rrb )_{\Gamma_{\ia\ib}} + (\llb v\rrb,\{ n\cdot \nabla w \})_{\Gamma_{\ia\ib}}
\\ \nonumber
&\qquad + \sum_{(i,j)\in\Iij} \glue h^{-1}(\llb v \rrb,\llb w\rrb)_{\Gamma_{\ia\ib}}
\end{align}
where $\glue$ is a positive parameter, and the linear functional $l$ is defined
\begin{align}
l(w) &= (f,w)_\Omega
\end{align}
The finite element method takes the form: find $u_h \in V_h/\IR$ such that
\begin{equation}
A_h(u_h,v) = l(v)\quad \forall v \in V_h
\label{eq:method}
\end{equation}
where the bilinear form $A_h$ is defined
\begin{align}
A_h(v,w) &= a_h(v,w) + j_h(v,w)
\end{align}
The stabilization form $j_h$ is defined
\begin{align}
j_h(v,w) = \sum_{\ia\in\Ii} j_{h,i}(\widehat{v},\widehat{w})
\end{align}
where $j_{h,i}$ is short for $j_{h,\Fhi}$ which is defined
\begin{equation}\label{eq:stab-F}
j_{h,\widehat{\mcF}}(\hatv,\hatw) = \sum_{k=1}^p \gamma_k h^{2k-1} (\llb \widehat{D}^k_n \widehat{v}\rrb ,\llb \widehat{D}^k_n \widehat{w}\rrb )_{\widehat{\mcF}}
\end{equation}
where $\widehat{\mcF}$ is a set of faces, $\{\gamma_k\}_{k=1}^p$ are positive parameters, and $\widehat{D}^k_n v$ on a face $\widehat{F}\in\widehat{\mcF}$ is the $k$:th derivative in the face normal direction to $\widehat{F}$ with respect to the Euclidean $\IR^2$ inner product.
We recall the definition of the jump and define the normal flux average over interfaces
\begin{equation}
\llb v \rrb = v_\ia - v_\ib, 
\qquad 
\lla n\cdot \nabla v \rra = (n_\ia \cdot \nabla v_\ia 
- n_\ib \cdot \nabla v_\ib)/2
\end{equation}
In the context of unfitted finite elements the stabilization form \eqref{eq:stab-F} was first analysed for linear elements in \cite{BH12} and extended to higher order elements in \cite{MassingLarsonLoggEtAl2013a}.

Note that, for $u|_{\Omega_i} \in H^{p+1}(\Omega_i)$, the finite element method (\ref{eq:method}) is consistent
and thus the error  $u - u_h$ satisfies the Galerkin orthogonality  
\begin{equation}\label{eq:galort}
A_h(u-u_h,v) = 0 \qquad \forall v \in V_h
\end{equation}

\begin{rem}[Penalty Parameter] The parameter $\glue$ is chosen large enough as in standard Nitsche type methods and suitable choices of the parameters 
$\{\gamma_k\}_{k=1}^p$ are provided in Section~\ref{sec:numerics} below. 
\end{rem}

\begin{rem}[Adaptation to Boundary Conditions]
While formulated above for a surface without boundary  the method \eqref{eq:method} is easily adapted to boundary conditions. For non-homogeneous Dirichlet and Neumann boundary conditions
\begin{alignat}{2}
u &= f_D \qquad &&\text{on $\partial\Omega_D$}
\\
n \cdot \nabla u &= f_N \qquad &&\text{on $\partial\Omega_N$}
\end{alignat}
where $\partial\Omega_D\cup\partial\Omega_D=\partial\Omega$ and $\partial\Omega_D\cap\partial\Omega_D=\emptyset$ we introduce the modified Nitsche forms
\begin{align}
\widetilde{A}_h(v,w) &= A_h(v,w)
+\sum_{\Gamma\in\mcB_{\partial\Omega_D}} 
\glue h^{-1}( v , w )_{\Gamma}
- ( n\cdot \nabla v , w )_{\Gamma}
- ( v , n\cdot \nabla w )_{\Gamma}
\\
\widetilde{l}(w)&= l(w)
+ \sum_{\Gamma\in\mcB_{\partial\Omega_N}}  
( f_N , w )_{\Gamma}
+ \sum_{\Gamma\in\mcB_{\partial\Omega_D}} 
( f_D , \glue h^{-1} w - n\cdot \nabla w )_{\Gamma}
\end{align}
and the resulting method reads: find $u_h\in V_h$ such that
\begin{equation}
\widetilde{A}_h(u_h,v) = \widetilde{l}(v)\quad \forall v \in V_h
\label{eq:method-bdry}
\end{equation}
Note that for a problem with Dirichlet boundary we no longer need $(u_h,1)_\Omega=(f,1)_\Omega=0$.
\end{rem}

\begin{rem}[Adaptation to Convection--Diffusion]
The method is also easily extended to cover convection--diffusion operators such as
\begin{align}
\mcL u = -\epsilon\Delta u + b\cdot\nabla u
\end{align}
where $\epsilon$ is a constant and $b|_{\Omega_i} = b_i :\Omega_i\rightarrow T_p(\Omega_i)$ is a tangential vector field.
In this case we get some additional terms and the bilinear form reads
\begin{align}
a_h(v,w) &= \sum_{\ia\in\Ii} \epsilon(\nabla v, \nabla w)_{\Omega_\ia}
+ (b\cdot\nabla v,w)_{\Omega_\ia}
\\ \nonumber
&\qquad -\sum_{(i,j)\in\Iij} 
\epsilon(\{ n\cdot \nabla v \},\llb w\rrb )_{\Gamma_{\ia\ib}}
+
\left( 2\{ n \cdot b v \}  , \{w\} + \gamma  \llb w \rrb \right)_{\Gamma_{\ia\ib}}
\\ \nonumber
&\qquad + \sum_{(i,j)\in\Iij} \glue \epsilon h^{-1} \left( \llb v \rrb,\llb w\rrb\right)_{\Gamma_{\ia\ib}}
\end{align}
where $\gamma$ is an upwind parameter and $\{\cdot\}$ for the non-flux terms is the usual average $\{w\}=\frac{w_i + w_j}{2}$.
Note that we require the vector field $b$ to be consistent across interfaces in the sense that $n_i\cdot b_i + n_j\cdot b_j = 0$ on $\Gamma_{ij}$.
\end{rem}

\subsection{Formulation in Reference Coordinates} 
In order to assemble the load vector and stiffness matrix we have the 
following expressions in reference coordinates
\begin{align}
( f , v )_{\Omega_\ia}
&=
\int_{\widehat{\Omega}_\ia} \widehat{f} \widehat{v} |\hatG|^{1/2} \, d\widehat{x}
\\
(\nabla v, \nabla w)_{\Omega_\ia}
&=
\int_{\hatOmega_\ia} \hatg(\hatG^{-1}\hatnabla \hatv, 
\hatG^{-1} \hatnabla \hatw)_{\hatOmega_\ia} |\hatG|^{1/2} \, d\widehat{x}
\\
&= \int_{\hatOmega_\ia} \hatnabla \hatv \cdot \hatG^{-1} 
\hatnabla \hatw  |\hatG|^{1/2} \, d\widehat{x}
\end{align}
For the interface terms on $\Gamma_{\ia \ib}$ we note that 
$p_{\ia\ib} = F_\ib^{-1} F_{\ia} : \hatGamma_{\ia\ib}^\ia 
\rightarrow  \hatGamma_{\ia\ib}^\ib$ is a bijection. Here we 
introduced the notation $\hatGamma_{\ia \ib}^\ia = F_\ia^{-1} (\Gamma_{\ia\ib})$. We may use $p_{\ia \ib}$ to pull back 
values from $\hatGamma_{\ia\ib}^\ib$ on $\hatGamma_{\ia\ib}^\ia$,
i.e. fetch values from the reference representation of $\Gamma_{\ia\ib}$ on $\hatOmega_j$
based on coordinates in the reference representation of $\Gamma_{\ia\ib}$ on $\hatOmega_i$.
We obtain the identity
\begin{align}
 \int_{\Gamma_{\ia \ib}} \llb v \rrb \llb w \rrb \, d\gamma
&= 
 \int_{\Gamma_{\ia \ib}} (v_\ia - v_\ib) (w_\ia - w_\ib ) \, d\gamma
 \\ \label{eq:interface-terms-a}
&= 
 \int_{\hatGamma_{\ia \ib}^i} ( \hatv_\ia - \hatv_\ib\circ p_{\ia\ib}^{-1}) (\hatw_\ia - \hatw_\ib \circ p_{\ia\ib}^{-1})  \|\hattau_\ia \|_{\hatg} \, d\hatgamma
\end{align}
and for the other interface term we get
\begin{align}\nonumber
&\int_{\Gamma_{\ia\ib}} \{ n\cdot \nabla v \} \llb w \rrb \, d\gamma
\\
&\qquad = \int_{\Gamma_{\ia\ib}} 
2^{-1}(n_\ia \cdot \nabla v_\ia - n_\ib \cdot \nabla v_\ib)(w_\ia - w_\ib) \, d\gamma
\\ \label{eq:interface-terms-b}
&\qquad=
\int_{\hatGamma_{\ia\ib}^i} 
2^{-1}(\hatg(\hatn_\ia, \widehat{\nabla v_\ia}) 
- \hatg(\hatn_\ib, \widehat{\nabla v_\ib})\circ p_{\ia\ib}^{-1}  ) (\hatw_\ia -\hatw_\ib\circ p_{\ia\ib}^{-1}) \|\hattau_\ia \|_{\hatg} \, d \hatgamma
\end{align}
where
\begin{align}\label{eq:local-a}
\hatg(\hatn_\ia, \widehat{\nabla v_\ia})
= \|\hatG^{-1} \hatnu_\ia\|_{\hatg}^{-1}
\hatnu_\ia \cdot \hatG^{-1}  \hatnabla \hatv_\ia
\end{align}
and $\hatnu_\ia$ is the unit normal to $\hattau_\ia$ with respect to 
the Euclidean $\IR^2$ inner product in $T(\hatOmega_\ia)$. 
To verify 
(\ref{eq:local-a}) we used the identity (\ref{eq:gradient-local})  for the gradient 
$\widehat{ \nabla v}_\ia = \hatG^{-1} \hatnabla \hatv_\ia$, and  the fact that the reference 
coordinates $\hatn_\ia$ of the normal $n_\ia$ may be expressed in 
terms of $\hatnu_\ia$,
\begin{equation}\label{eq:normal-coordinates}
\hatn_\ia = \frac{\hatG^{-1} \hatnu_\ia}{\| \hatG^{-1} \hatnu_\ia \|_{\hatg}} 
\end{equation}
which follows from the fact that $\frac{\hatG^{-1} \hatnu}{\| \hatG^{-1} \hatnu \|_{\hatg}}$ is a unit vector with respect to the metric inner product $\hatg$ 
which is also $\hatg-$orthogonal to the tangent vector $\hattau$,  
\begin{equation}
\hatg(\hatG^{-1}\hatnu,\hattau) = \hatnu \cdot \hattau = 0
\end{equation}

\subsection{Quadrature}

\paragraph{Quadrature on Cut Elements.}
To compute the terms implied from the above forms we generate a quadrature scheme for evaluation of integrals in the reference patches on the form
\begin{align}
\int_{\widehat{\Omega}_\ia \cap \hatK} f(\widehat{x}_1,\widehat{x}_2) \, d\widehat{x}
\label{eq:exint}
\end{align}
where the domain of integration is the intersection between a reference patch $\widehat{\Omega}_\ia$ and a finite element $\hatK$, and the integrand $f$ stems from tensor product polynomials of arbitrary order.
We denote this intersection
$\widehat{\omega} = \widehat{\Omega}_\ia \cap \hatK$
 and assume its boundary can be described by the union of $N$ non-overlapping curves $\widehat{\gamma}_k:[0,1]\rightarrow\mathbb{R}^2$, i.e.
\begin{align}
\partial\widehat{\omega} = \bigcup_{k=1}^N \widehat{\gamma}_k
\end{align}
where $\widehat{\gamma}_k(s) = \begin{bmatrix} \widehat{\gamma}_{k,1}(s) & \widehat{\gamma}_{k,2}(s) \end{bmatrix}^T$ is parametrized such that $\|\widehat{\gamma}_k'(s)\|_{\mathbb{R}^2} \neq 0$ and that $\widehat{\gamma}_k(s)$ in the positive direction traverses $\partial\widehat{\omega}$ counter-clockwise. Thus, the exterior unit normal to $\widehat{\omega}$, with respect to the $\mathbb{R}^2$ inner product, may be expressed
\begin{align}
\widehat{\nu} = \frac{1}{ \| \widehat{\gamma}_k' \|_{\mathbb{R}^2} } \begin{bmatrix}\widehat{\gamma}_{k,2}' \\ -\widehat{\gamma}_{k,1}' \end{bmatrix}
\label{eq:normal2d_mani}
\end{align}
We define a vector field
\begin{align} \label{eq:quadvecfield}
\widehat{\phi}: \widehat{\omega}\ni
\begin{bmatrix} \hatx_1 \\ \hatx_2 \end{bmatrix}
\rightarrow \begin{bmatrix} 0 \\ \int_a^{\widehat{x}_2} f(\widehat{x}_1,q) \, dq \end{bmatrix} \in \IR^2
\end{align}
where $a$ is an arbitrary constant which we choose as $a=\min_{\widehat{x}_2 \in \hatK} \widehat{x}_2$ and
note that we can express the integrand in \eqref{eq:exint} as $f = \widehat{\nabla}\cdot\widehat{\phi}$ by the fundamental theorem of calculus.
We rewrite \eqref{eq:exint} as two nested one dimensional integrals, one in each reference coordinate direction, by applying the divergence theorem in $\mathbb{R}^2$ and the following calculations
\begin{align}
\int_{\widehat{\omega}} f \, d\widehat{x} &= \int_{\widehat{\omega}} \widehat{\nabla}\cdot\widehat{\phi} \, d\widehat{x}
\\&=
\int_{\partial{\widehat{\omega}}} \widehat{\phi}\cdot\widehat{\nu} \, d\widehat{\gamma}
\\&=
\int_{\partial{\widehat{\omega}}}  \widehat{\nu}_2  \left( \int_{a}^{\widehat{x}_2} f(\widehat{x}_1,q) \, dq \right) \, d\widehat{\gamma}
\\&=
\sum_{k=1}^N \int_{{\widehat{\gamma}}_k} \widehat{\nu}_2 \left( \int_{a}^{\widehat{x}_2} f(\widehat{x}_1,q) \, dq \right) \, d\widehat{\gamma}
\\&=
\sum_{k=1}^N \int_0^1 \widehat{\nu}_2\circ\widehat{\gamma}_k(s)
\left( \int_{a}^{\widehat{\gamma}_{k,2}(s)} f(\widehat{\gamma}_{k,1}(s),q) \, dq \right)
 \| \widehat{\gamma}_k' \|_{\mathbb{R}^2} \, ds
\\&=
\sum_{k=1}^N \int_0^1
(-\widehat{\gamma}_{k,1}'(s))
\int_{a}^{\widehat{\gamma}_{k,2}(s)} f(\widehat{\gamma}_{k,1}(s),q) \, dq
\, ds
\label{eq:kgjns}
\\&=
\sum_{k=1}^N \int_0^1
\widehat{\gamma}_{k,1}'(s)\left(a - \widehat{\gamma}_{k,2}(s)\right)
\int_{0}^{1} f(\widehat{\gamma}_{k,1}(s),a + \tilde{s}(\widehat{\gamma}_{k,2}(s)-a)) \, d\tilde{s}
\, ds
\end{align}
where we in \eqref{eq:kgjns} use \eqref{eq:normal2d_mani} and in the last equality make a change of integration in the inner 1D integral.

Assuming the integrand $f$ is a tensor product polynomial of degree $p_f$, i.e. $f\in Q_{p_f}$, and that the boundary representation $\widehat{\gamma}_k(s)$ in each dimension is a polynomial of degree $p_\gamma$, i.e. $\widehat{\gamma}_k\in P_{p_\gamma}$, we can deduce the following resulting polynomial degrees of the integrands in the inner and outer 1D integral
\begin{alignat}{3}
& f(\widehat{\gamma}_{k,1}(s),a + \tilde{s}(\widehat{\gamma}_{k,2}(s)-a))
&&\in P_{p_f} \quad &&\text{w.r.t. $\tilde{s}$}
\\
&\widehat{\gamma}_{k,1}'(s)\left(a - \widehat{\gamma}_{k,2}(s)\right)
\int_{0}^{1} f(\widehat{\gamma}_{k,1}(s),a + \tilde{s}(\widehat{\gamma}_{k,2}(s)-a)) \, d\tilde{s}
\quad&&\in P_{2p_f p_\gamma + 2 p_\gamma - 1} \quad &&\text{w.r.t. $s$}
\end{alignat}
and thus we for each reference dimension can choose the number of Gauss points such that the 1D integrals in \eqref{eq:kgjns} are evaluated exactly. Let $\{s^{(i)},w^{(i)}\}_{i=1}^n$ and $\{\tilde{s}^{(j)},\tilde{w}^{(j)}\}_{j=1}^{\tilde{n}}$ be the set of Gauss quadrature points and weights which exactly integrates polynomials of degree $\leq p_f$ and degree $\leq 2p_f p_\gamma + 2 p_\gamma -1$ on $[0,1]$, respectively. Thus, the resulting quadrature points and weights are
\begin{align}
\begin{aligned}
\widehat{x}_1^{(ijk)} &=
\widehat{\gamma}_{k,1}(s^{(i)})
\\
\widehat{x}_2^{(ijk)} &=
a + \tilde{s}^{(j)}\left(\widehat{\gamma}_{k,2}(s^{(i)})-a\right)
\\
\widehat{w}^{(ijk)} &=
\widehat{\gamma}_{k,1}'(s^{(i)})\left(a - \widehat{\gamma}_{k,2}(s^{(i)})\right) w^{(i)} \tilde{w}^{(j)}
\end{aligned}
\qquad
\text{where}
\qquad
\left\{
\begin{aligned}
1 &\leq i \leq n \\ 1 &\leq j \leq \tilde{n} \\ 1 &\leq k \leq N
\end{aligned}
\right.
\end{align}
Note that the sum of the quadrature weights $\{\widehat{w}^{(ijk)}\}$ gives the area of $\widehat{\omega}$.
An illustration of the resulting quadrature points for an example intersection is given in Figure~\ref{fig:quadrature_mf} and in Table~\ref{tab:quadrature_mf} we list the polynomial degree and the number of integration points in the nested 1D integrals depending on the tensor product polynomial of the initial integrand $f$ and the boundary polynomial degree of the boundary representation.

\begin{rem}[Domain Complexity]
Note that the formulation of the quadrature rule assumes nothing about the complexity of the integration domain other than that its boundary should be well approximated by piecewise polynomial parametrizations. Thus, complex boundaries or holes pose no problem with this quadrature rule and the resolution of the boundary approximation is independent of the size of the finite elements.
On the other hand, allowing arbitrarily complex boundaries within an element means we cannot assume a readily available bulk description, for example a mesh, of the intersection between the element and the domain.
\end{rem}

\begin{rem}[Current Implementation] \label{rem:quadrature-implementation}
As the tensor product polynomials of our finite element basis functions will be perturbed by the Riemannian metric we compensate for this in the quadrature rule by choosing a higher order rule than indicated by the basis functions alone. Also, in cases where the trimmed patches in the reference domain are not exactly represented by piecewise $P_1$ curves we in our current implementation choose a $P_1$ representation with a resolution high enough for this error to be negligible. This use of $P_1$ representation is however not a limitation of the quadrature rule as seen in the above derivation and an alternative would be to use higher order approximations of the patch boundaries instead.
\end{rem}

\begin{rem}[Negative Quadrature Weights]
As seen in Figure~\ref{fig:quadrature_mf} the quadrature method includes both positive and negative weights which stems from adding and subtracting various parts of the integration domain. This is an undesirable property when considering reduced quadrature as inexact cancellation possibly could lead to loss of coercivity.
A possible modification which improves the method in this regard is to replace the constant lower bound $a$ in the integral in \eqref{eq:quadvecfield} by a polynomial of degree $\leq p_\gamma$ where the polynomial coefficients are chosen
such that the number of negative quadrature weights are minimized. If no restriction is placed on the polynomial coefficients, this could lead to some quadrature points being placed slightly outside the element.

However, in the present work we do further not investigate the aspect of reduced integration. We view this as a reference quadrature rule capable of integrating higher order tensor product polynomials and as noted in Remark~\ref{rem:quadrature-implementation} we rather use an increased integration order. In a complicated real world setting it is therefore advisable to chose an alternative quadrature scheme where positive quadrature weights can be guaranteed.
\end{rem}

\begin{figure}
\centering
\begin{subfigure}[b]{0.36\textwidth}\centering
\includegraphics[width=0.9\linewidth]{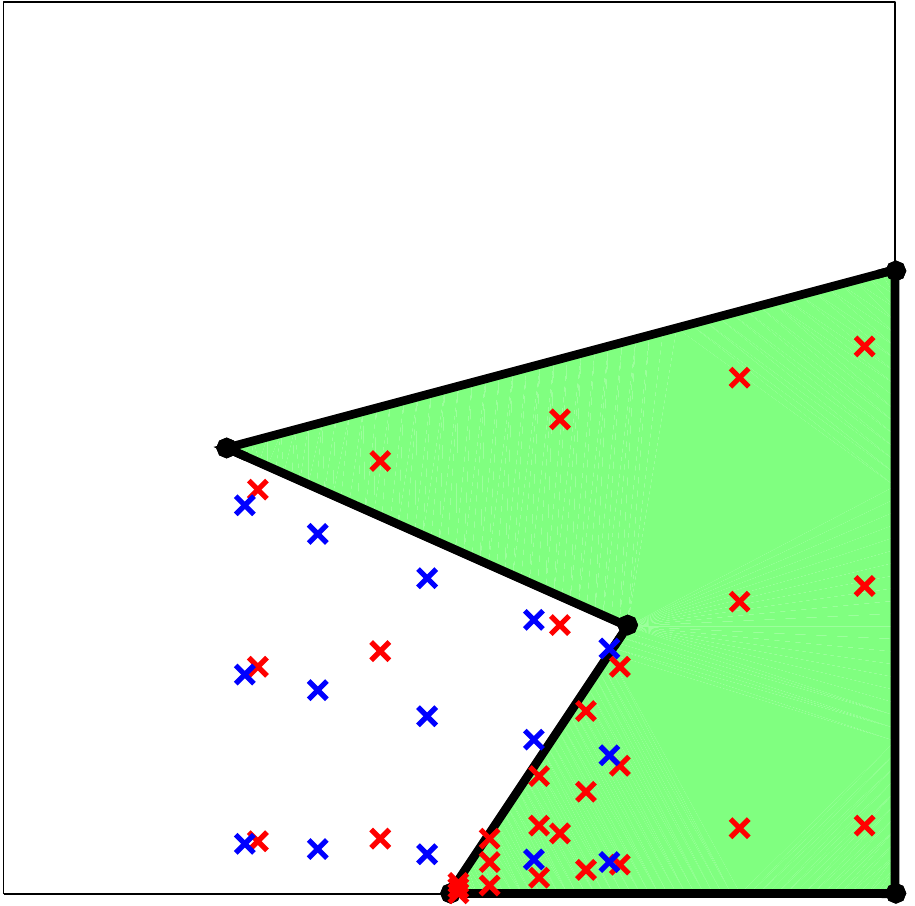}
\caption{$P_1$ approximation}
        \label{fig:p1q}
\end{subfigure}
\begin{subfigure}[b]{0.36\textwidth}\centering
\includegraphics[width=0.9\linewidth]{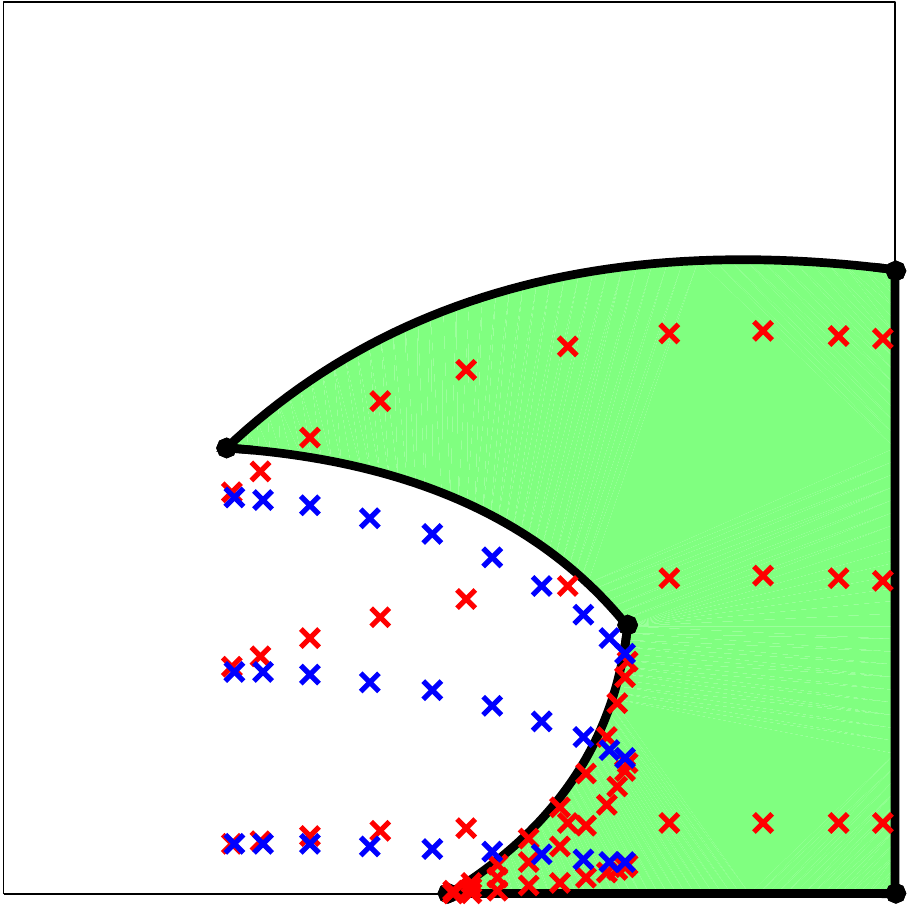}
\caption{$P_2$ approximation}
        \label{fig:p2q}
\end{subfigure}
\caption{Illustrations of quadrature rule for exact integration of $Q_4$ polynomials when the boundary is approximated using five $P_1$ segments (a) and five $P_2$ segments (b). Quadrature points with positive and negative weights are indicated in red and blue, respectively.}
\label{fig:quadrature_mf}
\end{figure}

\begin{table}[tb]
\centering
\begin{tabular}{llll}
\toprule
$(f,\widehat{\gamma}_k)$ & $\widehat{x}_2$ integrand
& $\widehat{x}_1$ integrand & points/seg. \\
\midrule
$(Q_2,P_1)$ & $P_2$ & $P_5$ & $2\times 3=6$\\
$(Q_4,P_1)$ & $P_4$ & $P_9$ & $3\times 5=15$\\
$(Q_6,P_1)$ & $P_6$ & $P_{13}$ & $4\times 7=28$\\
\midrule
$(Q_2,P_2)$ & $P_2$ & $P_{11}$ & $2\times 6=12$ \\
$(Q_4,P_2)$ & $P_4$ & $P_{19}$ & $3\times 10=30$\\
$(Q_6,P_2)$ & $P_6$ & $P_{27}$ & $4\times 14=56$\\
\bottomrule
\end{tabular}
\caption{Polynomial degree for nested 1D integrals and resulting number of quadrature points for each boundary segment assuming no optimizations, such as removing zero weight points, are used.}
\label{tab:quadrature_mf}
\end{table}

\begin{figure}[tb]
\centering
\includegraphics[width=0.27\linewidth]{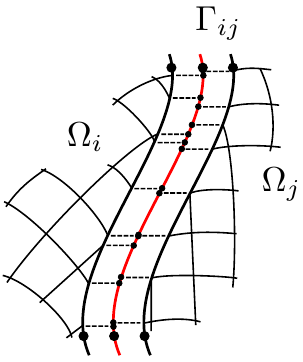}
\caption{Illustration of $\Gamma_{ij}$, i.e. the interface between patch $\domain_i$ and patch $\domain_j$. The points indicate the partition of the curve describing $\Gamma_{ij}$ such that each segment only is associated to a single element in $\widehat{\mcK}_{h,i}$ respectively in $\widehat{\mcK}_{h,j}$.}
\label{fig:mesh-interface}
\end{figure}

\paragraph{Quadrature on Interfaces.} To compute the interface terms 
(\ref{eq:interface-terms-a}) and (\ref{eq:interface-terms-b}) we construct a 
partition of $\hatGamma^\ia_{\ia\ib}$ which contains both all the 
intersection points between the curve $\hatGamma^\ia_{\ia \ib}$ 
and the mesh $\widehat{\mcK}_{h,\ia}$  as well as all the intersection points 
between the curve  $\hatGamma^\ib_{\ia \ib}$ and the mesh 
$\widehat{\mcK}_{h,\ib}$ mapped back to  $\hatGamma^\ia_{\ia\ib}$ using 
the mapping  $p_{\ia \ib}^{-1} = F_\ia^{-1} F_\ib$. Each interval in the partition of $\hatGamma^\ia_{\ia\ib}$ will thus be associated only with a single element in $\widehat{\mcK}_{h,\ia}$ and a single element in $\widehat{\mcK}_{h,\ib}$ and we apply a 1D Gauss quadrature rule on each interval. See Figure 
\ref{fig:mesh-interface} for an illustration of the partition of the interface.

\section{A Priori Error Estimates} \label{sec:error-est}

Let $a \lesssim b$ denote $a \leq C b$ with a constant $C$ independent of the mesh parameter $h$. 

\subsection{Norms}
Given a set of faces $\widehat{\mcF}$ in a mesh let
\begin{equation}
 \| \hatv \|^2_{j_{h,\widehat{\mcF}}} = j_{h,\widehat{\mcF}}(\widehat{v},\widehat{v}) = \sum_{k=1}^p \gamma_k h^{2k-1} \| \llb \widehat{D}^k_n \widehat{v}\rrb \|^2_{\widehat{\mathcal{F}}}
\end{equation}
We define the following energy norm
\begin{equation}
\tn v \tn^2_h  =\sum_{\ia\in\Ii}  \tn v \tn^2_{h, \ia} + \ \sum_{\mathclap{ (\ia,\ib)\in\Iij }} \ h \| \{ n \cdot  \nabla v \}  \|^2_{\interface_{\ia\ib}}  +  h^{-1} \|\llb v\rrb \|^2_{\interface_{\ia \ib}}
\label{eq:enNorm}
\end{equation}
where 
\begin{equation}
\tn v \tn^2_{h, \ia} =  \| \nabla v \|^2_{\Omega_\ia} + \| \hatv \|^2_{j_{h,\Fhi}}
\end{equation}

\subsection{Inverse Inequalities}

On elements which are partially outside the patch domain, as illustrated in Figure~\ref{fig:twoelms}, we
use the following inverse inequality to control a discrete function or its gradient on an element in terms of the gradient on a neighboring element and a suitable face term.

We will below make repeated use of the set of elements cut by the patch boundary $\Gamma_i$ and thus we define the set
\begin{align} \label{eq:gamma-partition}
\mcK_{h,i}(\Gamma_i) = \{ K \in \mcK_{h,i} : \Gamma_i \cap K \neq \emptyset \}
\end{align}
and analogously we define $\widehat{\mcK}_{h,i}(\widehat{\Gamma}_i)$ in the reference domain.

\begin{lem}\label{lem:inverse-ghost}
Let the two elements $\hatK_1,\hatK_2\in\widehat{\mcK}_{h,i}$ be neighbors of face $\widehat{F}$ with face normal $n_{\widehat{F}}$.
For all $\hatv \in { \hatV}_{h,i}|_{\hatK_1\cup \hatK_2}$ the following estimates then hold
\begin{align}\label{eq:inverse-ghost}
\| \hatnabla \hatv \|^2_{\hatK_1}
&\lesssim 
 \| \hatnabla \hatv \|^2_{\hatK_2} 
 + \| \hatv \|^2_{j_{h,\widehat{F}}}
\\ \label{eq:inverse-ghost-L2}
\| \hatv \|^2_{\hatK_1}
&\lesssim 
 \| \hatv \|^2_{\hatK_2} 
 + h^2 \| \hatv \|^2_{j_{h,\widehat{F}}}
\end{align}
\end{lem}

\begin{figure}
\centering
\includegraphics[width=0.25\linewidth]{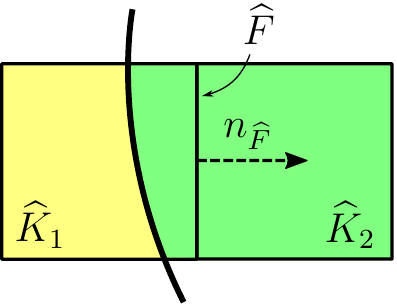}
\caption{Elements $\widehat{K}_1$ and $\widehat{K}_2$ neighboring face $\widehat{F}$ with the patch boundary cutting through $\widehat{K}_1$ such that $\widehat{K}_1$ is partially outside the patch domain.}
\label{fig:twoelms}
\end{figure}

\begin{proof}
We begin by proving estimate \eqref{eq:inverse-ghost} and we then make use of calculations in this proof when proving the second estimate \eqref{eq:inverse-ghost-L2}.
\paragraph{Proof of Estimate (\ref{eq:inverse-ghost}).}
Let $\hatv_l = \hatv |_{{\hatK}_l}$, $l=1,2.$ 
Since ${\hatv}_2$ is a 
polynomial we may evaluate $\hatv_2$ also  on $\hatK_1$. Using 
the triangle inequality followed by the inverse estimate $\|\hatv_2\| _{\hatK_1}
 \lesssim \|\hatv_2\|_{\hatK_2}$, we obtain 
\begin{align}
\| \hatnabla \hatv_1 \|_{\hatK_1} & \leq 
\| \hatnabla \hatv_2\|_{\hatK_1} + \| \hatnabla(\hatv_1 - \hatv_2) \|_{\hatK_1}
\\
&\lesssim 
\| \hatnabla \hatv_2\|_{\hatK_2} + \| \hatnabla{(\hatv_1 - \hatv_2)} \|_{\hatK_1}
\end{align}
To estimate the second term we note that using Taylor's formula on $(\hatv_1 - \hatv_2)$ at $x_{\widehat{F}} \in {\widehat{F}}$ in the face normal direction $n_{\widehat{F}}$ (with respect to the Euclidean $\IR^2$ inner product) gives
\begin{align}
(\hatv_1 - \hatv_2)(x)
&=
\underbrace{(\hatv_1 - \hatv_2)(x_{\widehat{F}})}_{=0}
+
\sum_{k=1}^p \hatD_{n}^k (\hatv_1-\hatv_2)(x_{\widehat{F}}) \frac{(n_{\widehat{F}}\cdot(x - x_{\widehat{F}}))^k}{k!} 
\end{align}
for $(x - x_F)  = s n_{\widehat{F}}$ with $s \in \IR$. Using an orthonormal coordinate system $\{\hate_1,\hate_2\}$ with $\hate_1=n_{\widehat{F}}$, i.e. a coordinate system which is aligned with the face normal and the face itself, we have
$x_{\widehat{F}} = [x_{\widehat{F},1}, x_2]^T$
and 
$n_{\widehat{F}}\cdot(x-x_{\widehat{F}}) = s = x_1 - x_{\widehat{F},1}$.
This gives the following expressions for the partial derivatives
\begin{align}
\hatpartial_1(\hatv_1- \hatv_2)(x) 
&= \sum_{k=1}^p \hatD_{n}^k (\hatv_1-\hatv_2)(x_{\widehat{F}}) \hatpartial_1\frac{(x_1  - x_{\widehat{F},1})^{k}}{k!}
\\
&= \sum_{k=1}^p \hatD_{n}^k (\hatv_1-\hatv_2)(x_{\widehat{F}}) \frac{(x_1 - x_{\widehat{F},1})^{k-1}}{(k-1)!} 
\\
\hatpartial_2(\hatv_1- \hatv_2) (x) 
&= \sum_{k=1}^p \hatpartial_2 \hatD_{n}^k (\hatv_1-\hatv_2)(x_{\widehat{F}}) \frac{(x_1  - x_{\widehat{F},1})^{k}}{k!} 
\end{align}

Using the Cauchy--Schwarz inequality for sums we obtain
\begin{align}
\|\hatpartial_1(\hatv_1- \hatv_2)\|^2_{\hatK_1} 
&\lesssim  \sum_{k=1}^p \|\hatD_{n}^k (\hatv_1-\hatv_2)\|_{\widehat{F}}^2 
\frac{1}{((k-1)!)^2} \int_0^h s^{2(k-1)} \, ds
\\
&\lesssim  \sum_{k=1}^p \| \llb \hatD_{n}^k \hatv \rrb \|_{\widehat{F}}^2 
\frac{h^{2k-1}}{((k-1)!)^2 (2 k - 1)}
\end{align}
and in the same way we have
\begin{align}
\|\hatpartial_2(\hatv_1- \hatv_2)\|^2_{\hatK_1} 
&\lesssim  \sum_{k=1}^p \|\hatpartial_2 \llb \hatD_{n}^k \hatv \rrb \|_{\widehat{F}}^2 
\frac{h^{2k+1}}{(k!)^2}
\\ \label{eq:inverselemma-B}
&\lesssim  \sum_{k=1}^p \| \llb \hatD_{n}^k \hatv \rrb \|_{\widehat{F}}^2 
\frac{h^{2k-1}}{(k!)^2}
\end{align}
where we used an inverse inequality in the last step to remove $\hatpartial_2$. In summary, we have
\begin{equation}
\|\hatnabla \hatv \|^2_{\hatK_1} 
\lesssim 
\|\hatnabla \hatv \|^2_{\hatK_2} 
+ 
 \sum_{k=1}^p h^{2k-1} \| \llb \hatD_{n}^k \hatv \rrb \|_{\widehat{F}}^2 
\end{equation}
which concludes the proof of estimate \eqref{eq:inverse-ghost}.

\paragraph{Proof of Estimate (\ref{eq:inverse-ghost-L2}).}
Starting in the same way as the proof of \eqref{eq:inverse-ghost} but without the gradients we have
\begin{align}
\| \hatv_1 \|_{\hatK_1} & \leq 
\| \hatv_2\|_{\hatK_1} + \| \hatv_1 - \hatv_2 \|_{\hatK_1}
\lesssim 
\| \hatv_2\|_{\hatK_2} + \| \hatv_1 - \hatv_2 \|_{\hatK_1}
\end{align}
As $(\hatv_1 - \hatv_2)|_{\widehat{F}}=0$ the Poincaré inequality holds yielding the following estimate
\begin{align}
\| \hatv_1 - \hatv_2 \|_{\hatK_1} \lesssim h^2 \| \nabla (\hatv_1 - \hatv_2) \|_{\hatK_1}
\end{align}
and we handle the remaining term as in the proof of \eqref{eq:inverse-ghost}.
\end{proof}

\begin{ass}[Patch Geometry]\label{ass:patches}
For a given element $K$ let  $\mcN_0(K) = K$ and for $l=1,2,\dots,$ let $\mcN_{l}(K)$ be 
the union of all elements that share a face or a node with an element in 
$\mcN_{l-1}(K)$, in other words $\mcN_l(K)$ is the set of elements that are 
neighbors of distance less or equal to $l$.  Assume that there is a positive 
integer $l$, a maximum mesh parameter $0<h_0$, and a positive constant
 $c$, such that for all $h \in (0,h_0]$ and all $K \in \mcK_{h}(\Gamma)$ there is an element $K' \in \mcN_{l}(K)$ such that $|K'\cap \hatOmega_\ia| \geq c|K|$.
\end{ass}

\begin{rem} This assumption limits the complexity of the reference subdomains 
$\hatOmega_\ia$. Note that the assumptions holds for $0< h \leq h_0$ with 
$h_0$ small enough when the boundary $\partial \hatOmega_\ia$ satisfies a cone condition. The assumption does not hold for instance in the vicinity of a 
cusp. In future work we will return to situations with more general patches  including very thin patches and patches with cusps since such patches may occur in CAD models used in practical engineering design. 
\end{rem}

\begin{lem}\label{eq:L2-bdry-elements} Given Assumption~\ref{ass:patches} it for $\hatK \in \widehat{\mcK}_{h,i}(\widehat{\Gamma}_i)$, holds
\begin{equation}\label{eq:inverse-ref}
\| \hatv \|^2_{\hatK} \lesssim \|\hatv \|^2_{\mcN_l(\hatK)\cap \hatOmega_\ia} 
+ \| \hatv \|^2_{j_{h,\mcF(\mcN_l(\hatK))}}
\end{equation}
where $\mcF(\mcN_l(\hatK))$ is the set of interior faces in the neighborhood $\mcN_l(\hatK)$.
\end{lem}
\begin{proof} This estimate follows directly from repeated use of Lemma 
\ref{lem:inverse-ghost} together with the Assumption on Patch Geometry 
that there is $\hatK' \in \mcN_l(\hatK)$ such that $|\mcN_l(\hatK) \cap \hatOmega_\ia| \gtrsim h^2$.
\end{proof}

\begin{lem} \label{lem:inverse}
There is a constant such that for all $v\in V_{h,\ia}$ 
it holds
\begin{equation}\label{eq:inverse}
h\| n_\ia \cdot \nabla v \|^2_{\Gamma_\ia} 
\lesssim \tn v \tn^2_{h, \ia}
\end{equation}
\end{lem}
\begin{proof} 
We proceed as follows
\begin{align}\label{eq:inv-a}
h \| n_\ia \cdot \nabla v \|^2_{\Gamma_\ia} 
&\leq
h \| \nabla v \|^2_{\Gamma_\ia}
\\ \label{eq:inv-b}
&= \sum_{K \in \mcK_{h,i}(\Gamma_\ia)}
h \|\nabla v \|^2_{\Gamma_\ia \cap K}
\\ \label{eq:inv-c}
&\lesssim   \sum_{K \in \mcK_{h,i}(\Gamma_\ia)}
h \|\hatnabla \hatv \|^2_{\hatGamma_\ia \cap \hatK}
\\ \label{eq:inv-d}
&\lesssim  \sum_{K \in \mcK_{h,i}(\Gamma_\ia)}
\|\hatnabla \hatv \|^2_{ \hatK}
\\ \label{eq:inv-e}
&\lesssim   \sum_{K \in \mcK_{h,i}(\Gamma_\ia)}
\Big( \|\hatnabla \hatv \|^2_{\mcN_l(\hatK)} 
+ \| \hatv \|^2_{j_h,\mcF(\mcN_l(\hatK))}\Big)
\\ \label{eq:inv-f}
&\lesssim   \sum_{K \in \mcK_{h,i}(\Gamma_\ia)}  \|\nabla v \|^2_{\mcN_l(K)} 
+  \| \hatv \|^2_{j_h,\mcF(\mcN_l(\hatK))}
\\ \label{eq:inv-g}
&\lesssim \|\nabla v\|^2_{\Omega_\ia} 
+ \|\hatv \|^2_{j_{h,\Fhi}}
\end{align}
where in (\ref{eq:inv-a}) we used the Cauchy--Schwarz inequality; 
in (\ref{eq:inv-b}) we divided the integral into element contributions; 
in (\ref{eq:inv-c}) we mapped to reference coordinates and used the 
bound
\begin{align}
\|\nabla v \|^2_{\Gamma_\ia \cap K} 
&= \int_{\hatGamma_\ia \cap \hatK} \|\hatG^{-1} \hatnabla \hat v\|^2_{\hatg} 
\|\tau\|_{\hatg}
 \, d\hatgamma
\\ 
& \leq \int_{\hatGamma_\ia \cap \hatK}  \lambda_\mathrm{max}(\hatG^{-1} ) \|\hatnabla \hat v\|^2_{\IR^2} \|\tau\|_{\hatg} 
 \, d\hatgamma
\\
&\leq 
\underbrace{ \left( \sup_{x \in \hatGamma_\ia \cap \hatK } \frac{\|\tau\|_{\hatg}}{\lambda_\mathrm{min}( \hatG ) }  \right)  }_{\lesssim 1}
 \int_{\hatGamma_\ia \cap \hatK} \|\hatnabla \hat v\|^2_{\IR^2} 
 \, d\hatgamma
\\&
\lesssim
\|\hatnabla \hatv \|^2_{\hatGamma_\ia \cap \hatK} 
\end{align}
In (\ref{eq:inv-d}) we used the following inverse trace inequality
\begin{equation}\label{eq:trace-internal}
h\| \hatnabla \hatv \|^2_{\hatGamma_\ia \cap \hatK} 
\lesssim \|\hatnabla \hatv \|^2_\hatK
\end{equation} 
which holds for $\hatv \in \hatV_{h,i}|_\hatK$ and we verify below;
in (\ref{eq:inv-e}) we used estimate (\ref{eq:inverse-ref}); 
in (\ref{eq:inv-f}) 
we used the bound
\begin{align}
\|\hatnabla \hatv \|^2_{\hat K} 
&= \int_{\hatK} \| \hatnabla \hatv \|^2_{\IR^2} \, d\hatx
\\&
\leq \int_{\hatK} \lambda_\mathrm{min}(\hatG^{-1}) \| \hatG^{-1} \hatnabla \hatv \|^2_\hatg \, d\hatx
\\&
\leq \underbrace{\left\| \bigl(|\hatG|^{1/2} \lambda_\mathrm{max}(\hatG) \bigr)^{-1} \right\|_{L^\infty(\hatK)}}_{\lesssim 1} 
\int_{\hatK} \| \widehat{\nabla v} \|_{\hatg}^2 |\hatG|^{1/2} \, d\hatx
\lesssim 
\|\nabla v \|^2_{K} 
\end{align}
for each of the elements in $\mcN_l(\hatK)$; and finally in \eqref{eq:inv-g} we used \eqref{eq:inverse-ref}.

\paragraph{Verification of (\ref{eq:trace-internal}).}  We have the bounds
\begin{equation}
h\| \hatnabla \hatv \|^2_{\hatGamma_\ia \cap \hatK} 
\leq 
h \underbrace{|\hatGamma_\ia \cap \hatK|}_{\lesssim h} \, 
\| \hatnabla \hatv \|^2_{L^\infty(\hatGamma_\ia \cap \hatK)} 
\leq
h^2 
\| \hatnabla \hatv \|^2_{L^\infty(\hatK)} 
\lesssim 
h^{-1} \| \hatnabla \hatv \|^2_{\hatK}
\end{equation}
where we used the fact that the length $\hatGamma_\ia \cap \hatK$ 
of the curve segment  $|\hatGamma_\ia \cap \hatK| \lesssim h$ for 
$h \in (0,h_0]$, with $h_0$ small enough, which holds since 
$\Gamma_\ia$ consists of a finite set of smooth curve segments,
and at last we used an inverse bound to estimate the $L^\infty$ norm in 
terms of the $L^2$ norm.
\end{proof}

\subsection{Coercivity and Continuity}
\begin{lem}[Coercivity and Continuity] The form $A_h$ satisfies:\\
(i) For large enough
$\glue$ it holds for all $v \in V_h$,
\begin{equation}\label{eq:coercivity}
\tn v \tn^2_h \lesssim A_h(v,v)
\end{equation}
(ii) For all  
$v,w \in V_h + H^{3/2+\epsilon}(\Omega)$ it holds
\begin{equation}\label{eq:continuity}
| A_h(v,w) | \lesssim \tn v \tn_h \tn w \tn_h 
\end{equation}
\end{lem}
\begin{proof} The first statement (i) follows directly from the inverse 
inequality (\ref{eq:inverse}) together with standard arguments, see 
for instance \cite{LarBen}. The second statement (ii) follows directly 
from the Cauchy--Schwarz inequality.
\end{proof}

\subsection{Interpolation}

Let $\widehat{\pi}_{h,\ia}: L^2(\widehat{\mcK}_{h,\ia}) \rightarrow 
\widehat{V}_{h,\ia}$ be the Scott--Zhang interpolation operator, see 
\cite{ScottZhang}. We recall the standard interpolation error estimate
\begin{equation}\label{eq:interpol-elem}
\|\hatu - \widehat{\pi}_{h,\ia} \hatu\|_{H^s(\widehat{K})} 
\lesssim 
h^{p+1 - s} \| \hatu \|_{H^{p+1}(\mcN(\widehat{K}))} 
\end{equation} 
where $\mcN(\hatK) \subset \hatmcK_{h,\ia}$ is the set of elements 
in $\hatmcK_{h,\ia}$ that are neighbors to $\hatK$. Next we note 
that for $0<h\leq h_0$, with $h_0$ small enough, we have 
\begin{equation}\label{eq:interpol-a}
\hatmcK_{h,\ia} \subset U_\delta(\hatOmega_\ia)
\end{equation}
where $U_\delta(\hatOmega_\ia) = \{x \in I^2 \,|\, d(x,\hatOmega_\ia)<\delta \}$, see Section \ref{sec:defsurface}. We 
define the global interpolation operator $\pi_h:L^2(\Omega) \rightarrow V_h$ 
as follows
\begin{equation}\label{eq:interpolant-global}
\widehat{(\pi_h v)}_\ia = \widehat{\pi}_{h,\ia} \hatv_\ia 
\end{equation}
where we used the fact that $\hatv_\ia$ is defined on $U_\delta(\hatOmega_\ia)$ and therefore the right hand side 
of (\ref{eq:interpolant-global}) is well defined due to (\ref{eq:interpol-a}).

\begin{lem}[Interpolation Error Estimate]\label{lem:interpol-energy} 
The interpolation operator $\pi_h$ defined by (\ref{eq:interpolant-global}) 
satisfies 
\begin{equation}\label{eq:interpol-energy}
\tn u - \pi_h u \tn_h \lesssim h^{p} \| \hatu \|_{H^{p+1}(\widehat{\Omega}_\ia)}
\end{equation}
\end{lem}
\begin{proof} Let $\rho = u - \pi_h u$. For each 
$\widehat{F} \in \widehat{\mcF}_{h,\ia}$ there are two neighboring 
elements $\widehat{K}_1$ and $\widehat{K}_2$ in $\widehat{\mcK}_{h,\ia}$. 
Using the triangle inequality followed by the trace inequality 
\begin{equation}\label{eq:trace-face}
\| \hatv \|^2_{\hatF} \lesssim h^{-1} \| \hatv \|_\hatK^2 + h \| \hatnabla \hatv \|^2_\hatK 
\qquad 
\hatv \in H^1(\hatK)
\end{equation}
and the interpolation estimate (\ref{eq:interpol-elem}) we obtain
\begin{align}
\| \widehat{\rho} \|_{j_h,\widehat{F}}^2
&=
\sum_{k=1}^p h^{2k-1} \|\llb \widehat{D}^k_n {\widehat{\rho}}\rrb\|^2_{\widehat{F}}
\\
&\leq 
\sum_{k=1}^p h^{2k-1} \sum_{j=1}^2 \|\widehat{D}^k_n {\widehat{\rho}_j}\|^2_{\widehat{F}} 
\\
&\lesssim 
\sum_{j=1}^2 \sum_{k=1}^p h^{2k-1}\Big( 
h^{-1} \|\widehat{D}^{k}_n {\widehat{\rho}_j}\|^2_{\widehat{K}_j} 
+ h \|\widehat{D}^{k+1}_n {\widehat{\rho}_j}\|^2_{\widehat{K}_j}\Big)
\\
&\lesssim 
\sum_{j=1}^2 \sum_{k=1}^p 
h^{2p} \| {\widehat{u}_j}\|^2_{H^{p+1}(\mcN(\widehat{K}_j))} 
\end{align}
Thus we conclude that
\begin{equation}
\| \widehat{\rho}\|^2_{j_{h,\ia}} \lesssim h^{2p} 
\| {\widehat{u}}\|^2_{H^{p+1}(\widehat{\Omega}_\ia)}
\end{equation}
Next we recall the following trace inequality 
\begin{equation}
h \| \hatv \|^2_{\hatK\cap \hatGamma_\ia}
\lesssim \| \hatv \|^2_\hatK + h^2 \| \hatnabla \hatv \|^2_\hatK 
\qquad \hatv \in H^1(\hatK),\;\; \hatK \in \hatmcK_{h,\ia} \label{eq:trace-cut2}
\end{equation}
which holds independent of the position of $\hatGamma_\ia$ 
in $\hatK$, see \cite{HanHanLar}. Introducing the notation 
\begin{equation}
\hatmcK_{h,\ia}(\hatGamma_\ia) 
= \{ \hatK \in \hatmcK_{h,\ia} : 
 \hatK \cap \hatGamma_\ia \neq \emptyset 
\} 
\end{equation}
we may estimate the remaining terms in the energy norm~\eqref{eq:enNorm} using 
the triangle inequality in~\eqref{proof:intp_eq1}, the trace inequality~\eqref{eq:trace-cut2} in \eqref{proof:intp_eq2}  and 
the interpolation estimate (\ref{eq:interpol-elem}) in~\eqref{proof:intp_eq3} as follows
\begin{align}
\tn \rho \tn_h^2 &\lesssim 
\sum_{\ia \in \Ii} \|\nabla \rho \|^2_{\Omega_\ia} + \|\widehat{\rho}\|^2_{j_{h,\ia}} 
+ h\|n_\ia \cdot \nabla \rho_\ia\|^2_{\Gamma_\ia} 
+ h^{-1} \|\rho_\ia\|^2_{\Gamma_\ia} \label{proof:intp_eq1}
\\
&\lesssim 
\sum_{\ia \in \Ii} \|\hatnabla \widehat{\rho} \|^2_{\hatOmega_\ia} + \|\widehat{\rho}\|^2_{j_{h,\ia}} 
+ h\|\hatnabla \widehat{\rho}_\ia\|^2_{\hatGamma_\ia} 
+ h^{-1} \|\widehat{\rho}_\ia\|^2_{\hatGamma_\ia}
\\
&\lesssim 
\sum_{\ia \in \Ii} \|\hatnabla \widehat{\rho} \|^2_{\hatOmega_\ia} 
+ \|\widehat{\rho}\|^2_{j_{h,\ia}} 
+ \|\hatnabla  \widehat{\rho}_\ia\|^2_{\widehat{\mcK}_{h,\ia}(\widehat{\Gamma}_\ia)} 
+ h^2\|\hatnabla^2 \widehat{\rho}_\ia\|^2_{\widehat{\mcK}_{h,\ia}(\widehat{\Gamma}_\ia)} \label{proof:intp_eq2}
\\ \nonumber 
&\qquad \quad
+ h^{-2} \|\widehat{\rho}_\ia\|^2_{\widehat{\mcK}_{h,\ia}(\widehat{\Gamma}_\ia)} 
+ \|\hatnabla \widehat{\rho}_\ia\|^2_{\widehat{\mcK}_{h,\ia}(\widehat{\Gamma}_\ia)} 
\\ \label{proof:intp_eq3}
&\lesssim h^{2p} \| \hatu \|_{H^{p+1}(\widehat{\Omega}_\ia)}^2   
\end{align}
which concludes the proof.
\end{proof}

\subsection{Error Estimates}

\begin{thm}[Energy Error Estimate]\label{thm:energy}
Let $u\in H^{p+1}(\Omega)$ be the solution to (\ref{eq:lb}) and $u_h$ 
the solution to (\ref{eq:method}), then 
\begin{align}\label{thmA:a}
\tn u - u_h \tn_h \lesssim h^{p} \sum_{i\in\mcI_\Omega} \| \hatu \|_{H^{p+1}(\widehat{\Omega}_\ia)}
\end{align}
\end{thm}
\begin{proof} Adding and subtracting the interpolant we have 
\begin{align}\label{thmA:b}
\tn u - u_h \tn_h &\leq \tn u - \pi_h u  \tn_h + \tn \pi_h u - u_h \tn_h
\\ \label{thmA:c}
&\lesssim h^{p} \| \hatu \|_{H^{p+1}(\widehat{\Omega}_\ia)} + \tn \pi_h u - u_h \tn_h
\end{align}
where we used the interpolation error estimate (\ref{eq:interpol-energy}). For the 
second term we have, using the notation $e_h = \interp_h u - u_h$ and Galerkin orthogonality (\ref{eq:galort}), 
\begin{align}\label{thmA:d}
\tn e_h \tn_h^2 &\lesssim A_h(e_h,e_h) 
\\ \label{thmA:e}
&= A_h(\pi_h u - u,e_h)
\\ \label{thmA:f}
&\lesssim 
\tn \pi_h u - u \tn_h \tn e_h \tn_h 
\end{align}
Thus we conclude that 
\begin{align}\label{thmA:g}
\tn e_h \tn_h \lesssim \tn \pi_h u - u \tn_h 
\lesssim h^{p} \| \hatu \|_{H^{p+1}(\widehat{\Omega}_\ia)}
\end{align}
where we finally used the interpolation estimate (\ref{eq:interpol-energy}) again. 
Together (\ref{thmA:c}) and (\ref{thmA:g}) concludes the proof.
\end{proof}

\begin{thm}[${\boldsymbol L}^{\boldsymbol 2}$ Error Estimate]\label{thm:L2}
Let $u$ be the solution to (\ref{eq:lb}) and $u_h$ 
the solution to (\ref{eq:method}), then 
\begin{equation}
\| u - u_h \|_\Omega \lesssim h^{p+1} \| \hatu \|_{H^{p+1}(\widehat{\Omega}_\ia)}
\end{equation}
\end{thm}
\begin{proof} Let $e = u - u_h$ be the error, and $\phi$ the 
solution of the dual problem
\begin{subequations}
\begin{alignat}{2}
-\Delta \phi &= e
\qquad &&\text{in all $\Omega_\ia \in \Patches$}
\\
\left\llb \phi\right\rrb  &= 0 
\qquad &&\text{on all $\Gamma_{\ia\ib} \in \Interfaces$}
\\
{n}_{\ia} \cdot \nabla \phi_\ia +{n}_{\ib} \cdot \nabla \phi_\ib  &= 0
\qquad &&\text{on all $\Gamma_{\ia\ib} \in \mcG$}
\end{alignat}
\end{subequations}
Recall that $\int_{\Omega} u \, dx = \int_\Omega u_h \, dx = 0 \, dx$ and thus $\int_\Omega e =0 \, dx$ and we conclude that the dual problem has a unique solution in $V$, see 
(\ref{eq:spaceV}), that satisfies 
\begin{equation}\label{eq:ellregdual}
\sum_{i\in\mcI_\Omega}\| \phi \|_{H^2(\Omega_i)} \lesssim \|e\|_\Omega
\end{equation}
Multiplying the dual problem by $e$, integrating by parts 
on each subdomain $\Omega_\ia$, and using the interface conditions on $\phi$ we obtain
\begin{align}
\|e \|^2_\Omega
&= \sum_{\ia \in \Ii}  -(e ,\Delta \phi)_{\Omega_i}
\\&
= a_h(e,\phi)
\\& \label{eq:kdjbn}
= a_h(e,\phi - \pi_h \phi) - j_h(e,\pi_h \phi)
\\& \label{eq:ibfdu}
\leq
\| e \|_{a_h} \| \phi -\pi_h\phi \|_{a_h} + \| e \|_{j_h} \| \pi_h\phi \|_{j_h}
\\&
\leq
\tn e \tn_{h} \bigl( \underbrace{\| \phi -\pi_h\phi \|_{a_h}}_{I} + \underbrace{\| \pi_h\phi \|_{j_h}}_{II} \bigr)
\lesssim h^p (I + II)
\end{align}
where we in \eqref{eq:kdjbn} used the Galerkin orthogonality \eqref{eq:galort} to subtract $A_h(e,\pi_h\phi)=0$, and we in \eqref{eq:ibfdu}  applied the Cauchy--Schwarz inequality and also introduced the norms $\| v \|_{a_h}^2 = a_h(v,v)$ and $\| v \|_{j_h}^2 = j_h(v,v)$ induced by their respective forms.
In the final inequality we applied the energy norm estimate Theorem~\ref{thm:energy}.

\paragraph{Estimate of $\boldsymbol I$.}
Using the triangle inequality on the jumps and averages, applying the trace inequality and standard interpolation estimates we obtain
\begin{align}
I &
\lesssim h \sum_{i\in\mcI_\Omega} \| \phi \|_{H^2(\Omega_i)}
\lesssim h \| e \|_\Omega
\end{align}
where we in the last inequality use the elliptic regularity of the dual solution \eqref{eq:ellregdual}.

\paragraph{Estimate of $\boldsymbol I \boldsymbol I$.}

Consider two neighboring elements $\hatK_1,\hatK_2\in\widehat{\mcK}_{h,i}$ sharing face $\widehat{F}$.
Introducing a patchwise interpolant $\widehat{\pi}_F:L^2(\hatK_1 \cup \hatK_2)\rightarrow Q^p(\hatK_1 \cup \hatK_2)$, where $Q^p$ is the space of tensor product polynomials of degree $\leq p$, we note that we can subtract $\widehat{\pi}_{F}\widehat{\phi}$ inside the stabilization terms as it will give no contribution due to the jump over faces.
We proceed as follows
\begin{align}
h^{2k-1} \bigl\| \llb \widehat{D}^k_n \widehat{\pi}_{h,i}\widehat{\phi} \rrb \bigr\|_{\widehat{F}}^2
&=
h^{2k-1} \bigl\| \llb \widehat{D}^k_n ( \widehat{\pi}_{h,i}\widehat{\phi} - \widehat{\pi}_{F}\widehat{\phi} ) \rrb \bigr\|_{\widehat{F}}^2
\\&\lesssim \label{eq:l2a}
h^{2k-2} \bigl\|  \widehat{\pi}_{h,i}\widehat{\phi} - \widehat{\pi}_{F}\widehat{\phi} \bigr\|_{H^k(\widehat{K}_1 \cup \widehat{K}_2)}^2
\\&\lesssim \label{eq:l2b}
\bigl\|  \widehat{\pi}_{h,i}\widehat{\phi} - \widehat{\pi}_{F}\widehat{\phi} \bigr\|_{H^1(\widehat{K}_1 \cup \widehat{K}_2)}^2
\\&\leq \label{eq:l2c}
\left(
\bigl\|  \widehat{\pi}_{h,i}\widehat{\phi} - \widehat{\phi} \bigr\|_{H^1(\widehat{K}_1 \cup \widehat{K}_2)}^2
+
\bigl\|  \widehat{\phi} - \widehat{\pi}_{F}\widehat{\phi}  \bigr\|_{H^1(\widehat{K}_1 \cup \widehat{K}_2)}^2
\right)
\\&\lesssim \label{eq:l2d}
h^2 \| \widehat{\phi} \|_{H^2(\widehat{K}_1 \cup \widehat{K}_2)}^2
\end{align}
where in \eqref{eq:l2a} we used an inverse trace inequality, in \eqref{eq:l2b} we used an inverse inequality, in \eqref{eq:l2b} we added and subtracted $\widehat{\phi}$ and used the triangle inequality, and finally in \eqref{eq:l2d} we used interpolation estimates.
We thus have the estimate
\begin{align}
II^2 =
\| \pi_h\phi \|_{j_h}^2
\lesssim h^2 \sum_{i\in\mcI_\Omega} \sum_{\widehat{F}\in\mcF(\widehat{\mcK}_{h,i})}\| \widehat{\phi} \|_{H^2(\widehat{K}_1 \cup \widehat{K}_2)}^2
\lesssim h \sum_{i\in\mcI_\Omega} \| \phi \|_{H^2(\mcK_{h,i})}
\lesssim h  \| e \|_{\Omega}
\end{align}
which concludes the proof.
\end{proof}

\subsection{Condition Number Estimate}

To prove an upper bound on the stiffness matrix condition number we follow the approach in \cite{BH12,Ern06}.
Let $\{\varphi_i\}_{i=1}^N$ be the standard piecewise tensor product polynomial Lagrange basis functions associated with the nodes in $\mcK_h$ and let $\mcA$ and $\mcM$ be the stiffness and mass matrices with elements $\mcA_{ij}=A_h(\varphi_j,\varphi_i)$ and $\mcM_{ij}=(\varphi_j,\varphi_i)_\Omega$, respectively. The condition number for the stiffness matrix is defined by
\begin{align}\label{eq:cond-def}
\kappa(\mcA) = |\mcA|_{\IR^N}|\mcA^{-1}|_{\IR^N}
\end{align}
where $|\cdot |_{\IR^N}$ on matrices denotes the operator norm 
\begin{align}
|\mcA|_{\IR^N} = \sup_{\mathbf{V}\in \IR^N\backslash 0} \frac{| \mcA\mathbf{V}|_{\IR^N}}{| \mathbf{V} |_{\IR^N}}
\end{align}
and $|\cdot |_{\IR^N}$ on vectors denotes the Eucledian norm.

\begin{thm}[Upper Bound on Condition Number] \label{thm:condition-number}
The condition number of the stiffness matrix $\mcA$ satisfies the estimate
\begin{equation}
\kappa(\mcA) \lesssim h^{-2}
\end{equation}
for all $h\in(0,h_0]$ with $h_0$ sufficiently small.
\end{thm}

\begin{proof}
If $v=\sum_{i=1}^N \mathbf{V}_i\varphi_i$ and $\{\varphi_i\}_{i=1}^N$ is the usual nodal basis on $\mcK_h$ the following well known estimate holds
\begin{align} \label{eq:discrete-equiv}
h | \mathbf{V} |_{\IR^N} \lesssim \| v \|_{\mcK_{h}} \lesssim h | \mathbf{V} |_{\IR^N}
\end{align}

We will make use of the inverse inequality
\begin{align}\label{eq:energy-L2-inverse}
\tn v \tn_h \lesssim h^{-1} \| v \|_{\mcK_h} \qquad \forall v \in V_h/\IR
\end{align}
and the discrete Poincaré inequality
\begin{align}\label{eq:poincare}
\| v \|_{\mcK_{h}} \lesssim \tn v \tn_h \qquad \forall v \in V_h/\IR
\end{align}
The inverse inequality \eqref{eq:energy-L2-inverse} is proven by first applying the triangle inequality to all jump and average terms and then using Lemma~\ref{lem:inverse} on the consistency terms, the inverse inequality
\begin{align}
h^{-1} \| v \|_{\Gamma_i \cap K}^2 \lesssim \| v \|_K^2
\end{align}
on the jump penalty term and the inverse inequality
\begin{align}
h^{2k-1}\| \widehat{D}_n^k \widehat{v} \|_{\partial\widehat{K}}^2
\lesssim
h^{2k-2} \| \widehat{v} \|_{H^k(\widehat{K})}^2
\lesssim
h^{-2} \|\widehat{v}\|_{\widehat{K}}^2
\end{align}
on the stability terms.
The discrete Poincaré inequality \eqref{eq:poincare} is proven by first separating the cut elements and applying Lemma~\ref{eq:L2-bdry-elements} which gives
\begin{align}
\| v \|_{\mcK_{h}}^2
\lesssim
\| v \|_{\Omega}^2 + \sum_{i\in\mcI_\Omega} \| v \|_{\mcK_{h,i}(\Gamma_i)}^2
\lesssim
\| v \|_{\Omega}^2 + j_h(v,v)
\lesssim
\| \nabla v \|_{\Omega}^2 + j_h(v,v)
\lesssim
\tn v \tn_h^2
\end{align}
where we in the second last inequality apply the standard Poincaré inequality on the first term.

We now turn to estimating $|\mcA|_{\IR^N}$ and $|\mcA^{-1}|_{\IR^N}$ separately. The product of these estimates will give a bound on the condition number $\kappa(\mcA)$ by its definition \eqref{eq:cond-def}.

\paragraph{Estimate of $|\mcA|_{\IR^N}$.}
Let $\mathbf{V} \in \widetilde{\IR^N}$ where $\widetilde{\IR^N} = \{ \mathbf{V}\in\IR^N \, : \, \mathbf{V}^T \mcM \mathbf{V} = 0 \}$. In other words $\widetilde{\IR^N}$ is the space of coefficient vectors corresponding for discrete functions in $V_h/\IR$.
By the definition of the method (in matrix form) and using continuity \eqref{eq:continuity} we have
\begin{align}
| \mcA\mathbf{V} |_{\IR^N}
&=
\sup_{\mathbf{W}\in \IR^N\backslash 0} \frac{(\mathbf{W},\mcA\mathbf{V})_{\IR^N}}{| \mathbf{W} |_{\IR^N}}
\\&=
\sup_{\mathbf{W}\in \widetilde{\IR^N}\backslash 0} \frac{(\mathbf{W},\mcA\mathbf{V})_{\IR^N}}{| \mathbf{W} |_{\IR^N}}
\\&
=
\sup_{w \in \{V_h/\IR\}\backslash 0} \frac{A_h(v,w)}{| \mathbf{W} |_{\IR^N}}
\\&
\lesssim
\sup_{w \in \{V_h/\IR\}\backslash 0} \frac{\tn v \tn_h \tn w \tn_h}{| \mathbf{W} |_{\IR^N}}
\\&
\lesssim |\mathbf{V} |_{\IR^N}
\end{align}
where we in the last inequality used the inverse estimate \eqref{eq:energy-L2-inverse} and \eqref{eq:discrete-equiv}.
It follows that
\begin{align} \label{eq:stiffness-bound}
|\mcA |_{\IR^N} \lesssim 1
\end{align}

\paragraph{Estimate of $|\mcA^{-1}|_{\IR^N}$.}
Let $\mathbf{V}\in\widetilde{\IR^N}$.
Using \eqref{eq:discrete-equiv}, the Poincaré inequality \eqref{eq:poincare}, coercivity \eqref{eq:coercivity}
and the Cauchy--Schwarz inequality we obtain
\begin{align}
h | \mathbf{V} |_{\IR^N}
\lesssim
\| v \|_{\mcK_{h}}
\lesssim
\tn v \tn_h
\lesssim
\frac{A_h(v,v)}{\tn v \tn_h}
=
\frac{(\mcA\mathbf{V},\mathbf{V})_{\IR^N}}{\tn v \tn_h}
\lesssim
\frac{|\mcA\mathbf{V}|_{\IR^N}|\mathbf{V}|_{\IR^N}}{\tn v \tn_h}
\lesssim
h^{-1} |\mcA\mathbf{V}|_{\IR^N}
\end{align}
Since $\mathbf{V}$ is arbitrary we by choosing $\mathbf{U}=\mcA\mathbf{V}\in\widetilde{\IR^N}$, i.e. $\widetilde{\IR^N} \ni \mathbf{V}=\mcA^{-1}\mathbf{U}$, get
\begin{align}
|\mcA^{-1}|_{\IR^N} \lesssim h^{-2}
\end{align}
which in combination with \eqref{eq:stiffness-bound} concludes the proof.
\end{proof}

\section{Numerical Results} \label{sec:numerics}

In this section we present our numerical experiments to verify convergence rates and the stability of the cut finite element method on patchwise parametrized surfaces. We also provide various numerical examples. 

\subsection{Model Problems} \label{sec:modprob}
For our convergence and stability results we choose the same Laplace--Beltrami model problems as in \cite{OMRAG09}; a problem on the unit sphere and a problem on a torus surface. The solutions and load functions to these problems satisfy $(u,1)_\domain = (f,1)_\domain = 0$.  
\paragraph{Surface and Analytical Solution.}
The surfaces and analytical solutions for our two model problems are illustrated in Figure~\ref{fig:analytical} and described below.
\begin{itemize}
\item \textbf{Sphere:}
The surface $\domain$ is the unit sphere centered in origo and we use a manufactured problem with analytical solution  $u = 3x^2 y-y^3$.

\item \textbf{Torus:}
The surface $\Omega$ is a torus with inner radius $r = 0.6$ and outer radius $R = 1$. This surface can be expressed in Cartesian coordinates as the points
\begin{equation}
\left\{
  x = (R + r \cos \theta) \cos{\phi} \, , \
  y = (R + r \cos \theta) \sin{\phi} \, , \
  z = r \sin \theta
\right\}
\label{eq:torusPar}
\end{equation}
for  $0 \leq \theta \leq 2 \pi$ and $0 \leq \phi < 2 \pi$ where $\{\theta,\phi\}$ are toroidal coordinates of the surface.
We use a manufactured problem with analytical solution $ u = \sin( 3 \phi) \cos(3 \theta + \phi)$. 
\end{itemize}

\begin{figure}
    \centering
    \begin{subfigure}[b]{0.47\textwidth}
        \centering
        \includegraphics[width=0.75\textwidth]{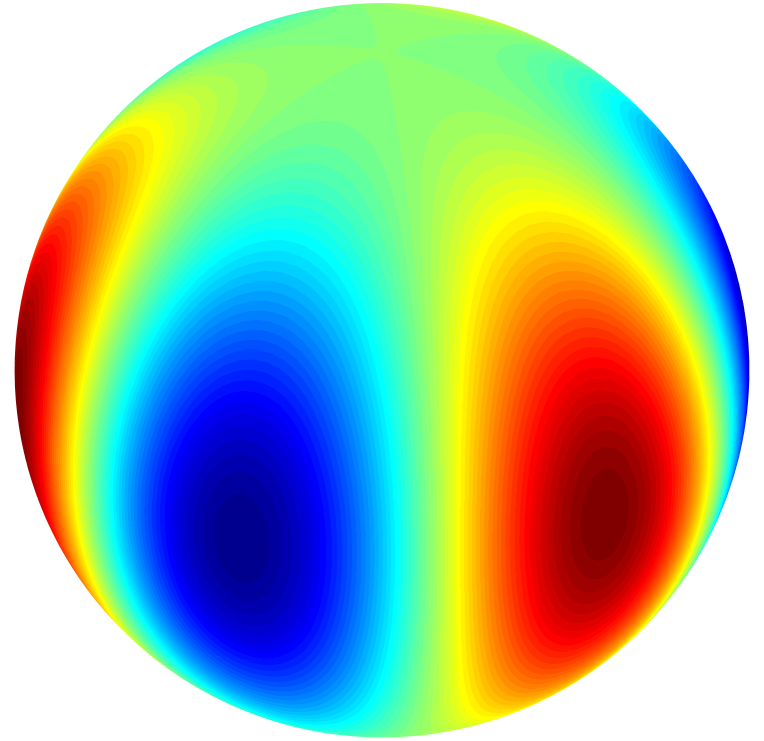}
        \caption{Sphere}
    \end{subfigure}
    \begin{subfigure}[b]{0.47\textwidth}
        \centering
        \includegraphics[width=0.9\textwidth]{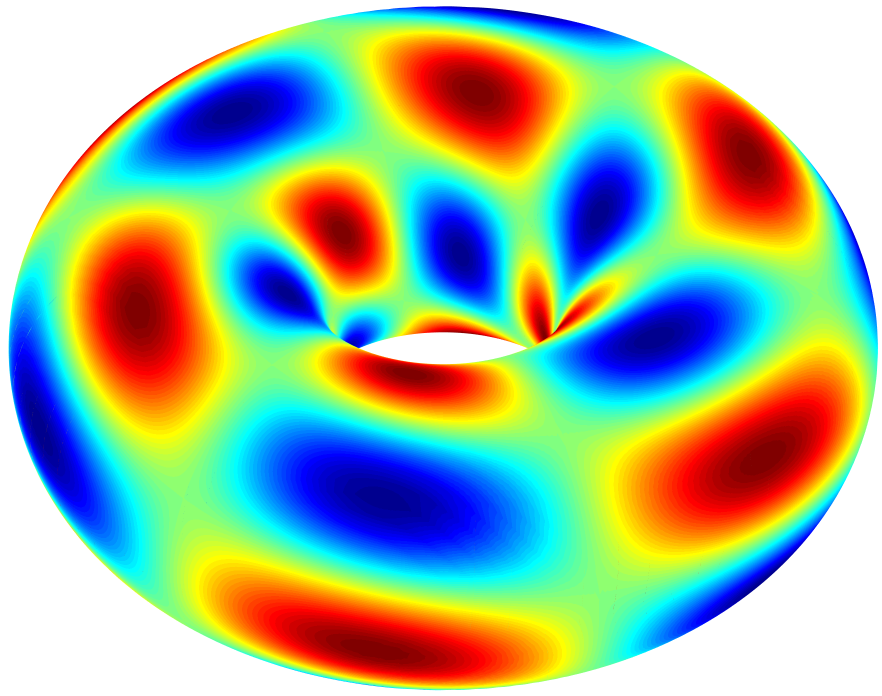}
        \caption{Torus}
    \end{subfigure}
    \caption{Analytical solutions to the two model problems.}
    \label{fig:analytical}
\end{figure}

\paragraph{Patchwise Surface Description.}
In the presented method the surface $\Omega$ is described by a set of mappings $\{ F_i \}_{i\in\mcI_\Omega}$ and trimmed patches in reference coordinates $\{ \hatOmega_i \}_{i\in\mcI_\Omega}$ such that $\Omega = \bigcup_{i\in\mcI_\Omega} F_i(\hatOmega_i)$ and $\bigcap_{i\in\mcI_\Omega} F_i(\hatOmega_i) = \emptyset$.
To construct such a description for the two model problems we first create a closed surface approximation of $\Omega$ consisting of a number of polygons $\{T_i\}_{i\in\mcI_\Omega}$. For each polygon $T_i$
we by a simple affine mapping can create an inverse mapping down to a reference patch $\hatOmega_i$ in $[0,1]^2$. To map onto the surface $\Omega$ we from $T_i$ use a
closest point mapping and by combining the inverse mapping and the closest point mapping we define $F_i:\hatOmega_i\rightarrow \Omega_i$. The actual patchwise descriptions used for the model problem are illustrated in Figure~\ref{fig:model-patches}.

\begin{figure}
    \centering
    \begin{subfigure}[b]{0.47\textwidth}
        \centering
        \includegraphics[width=0.75\textwidth]{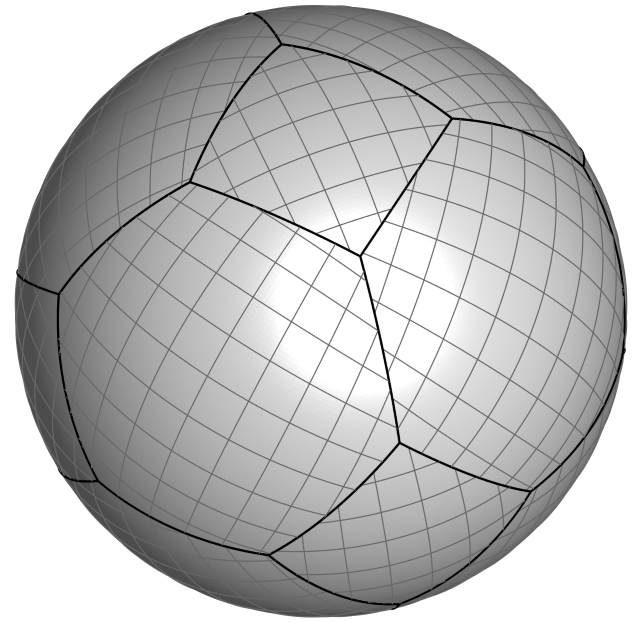}
       \caption{Sphere}
    \end{subfigure}
    \begin{subfigure}[b]{0.47\textwidth}
        \centering
        \includegraphics[width=0.9\textwidth]{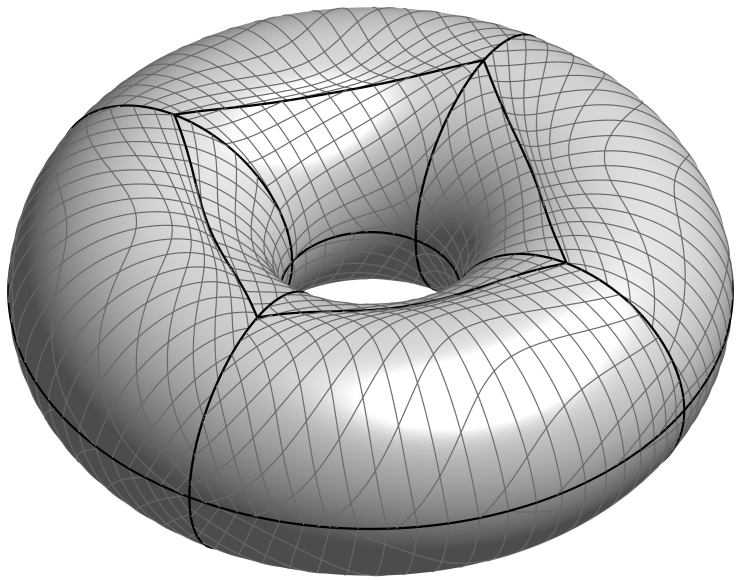}
        \caption{Torus}
    \end{subfigure}
      \caption{Patchwise described surfaces for the two model problems. The patches are displayed with coarse meshes to illustrate the effect of the constructed mappings.}
    \label{fig:model-patches}
\end{figure}

\subsection{Implementation Aspects}

We use tensor product Lagrange finite elements of order $p$ on quadrilaterals in our implementation. In the results below we for the Nitsche interface terms used the parameter $\beta=100$ and for the CutFEM stability terms used parameters $\gamma_k = 10^{-2}$, $k=1,\dots,p$. The latter choice is numerically investigated in Section~\ref{sec:stability} below. To impose the average constraint $(u_h,1)_\Omega=0$ we use a Lagrange multiplier approach, see for example \cite{LarBen}.

\subsection{Convergence}

\begin{figure}
    \centering
    \begin{subfigure}[b]{0.47\textwidth}
        \centering
        \includegraphics[width=0.75\textwidth]{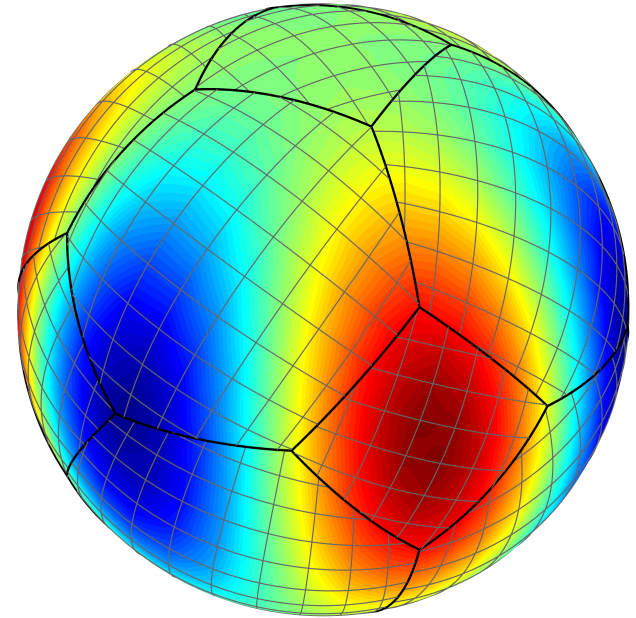}
       \caption{Sphere}
    \end{subfigure}
    \begin{subfigure}[b]{0.47\textwidth}
        \centering
        \includegraphics[width=0.9\textwidth]{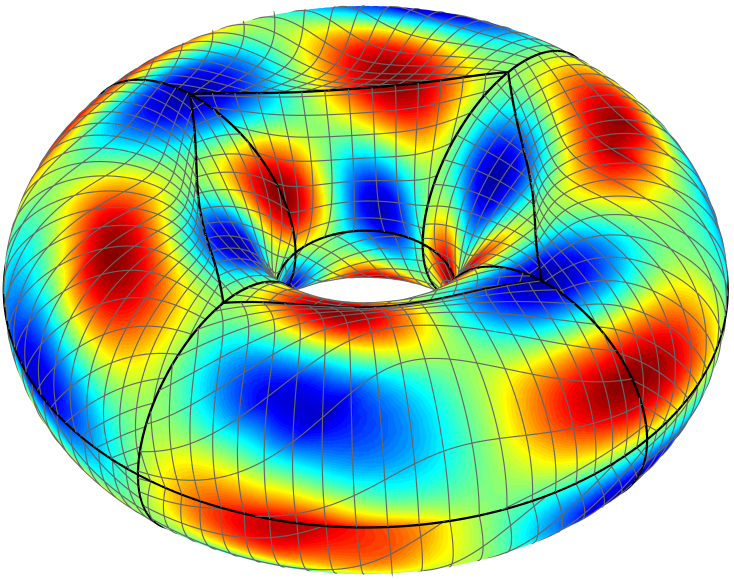}
        \caption{Torus}
    \end{subfigure}
      \caption{Finite element solutions $u_h$ to the model problems.}
    \label{fig:u_h}
\end{figure}

To confirm our theoretical results in Theorem~\ref{thm:energy} and Theorem~\ref{thm:L2} we present convergence results for the energy norm error and $L^2$ norm error in Figure~\ref{fig:conv-energy} and Figure~\ref{fig:conv-L2}, respectively. Example numerical solutions for the two model problems are displayed in Figure~\ref{fig:u_h}. In these studies the geometry representation, i.e. the reference patches $\hatOmega_i$ and mappings $F_i$, is kept fixed while the background grid is refined.

\begin{figure}
\centering
\begin{subfigure}[b]{0.48\textwidth}\centering
\includegraphics[width=0.9\textwidth]{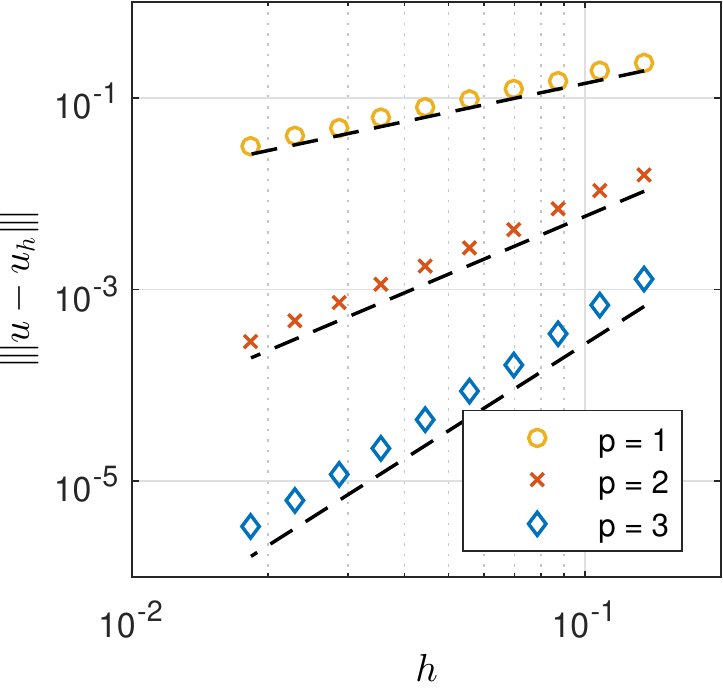}
\caption{Sphere}
\end{subfigure}  
\begin{subfigure}[b]{0.48\textwidth}\centering
\includegraphics[width=0.9\textwidth]{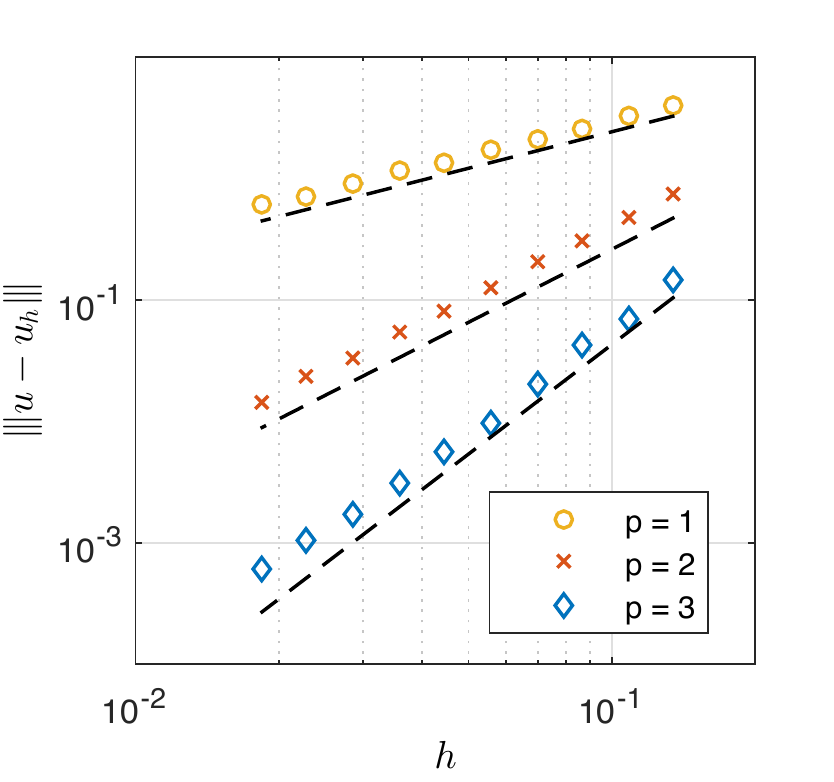}
\caption{Torus}
\end{subfigure}
\caption{Convergence results in the energy norm for the two model problems. The dashed reference lines are $h^{p}$.}
\label{fig:conv-energy}
\end{figure}

\begin{figure}
\centering
\begin{subfigure}[t]{0.48\textwidth}\centering
\includegraphics[width=0.9\textwidth]{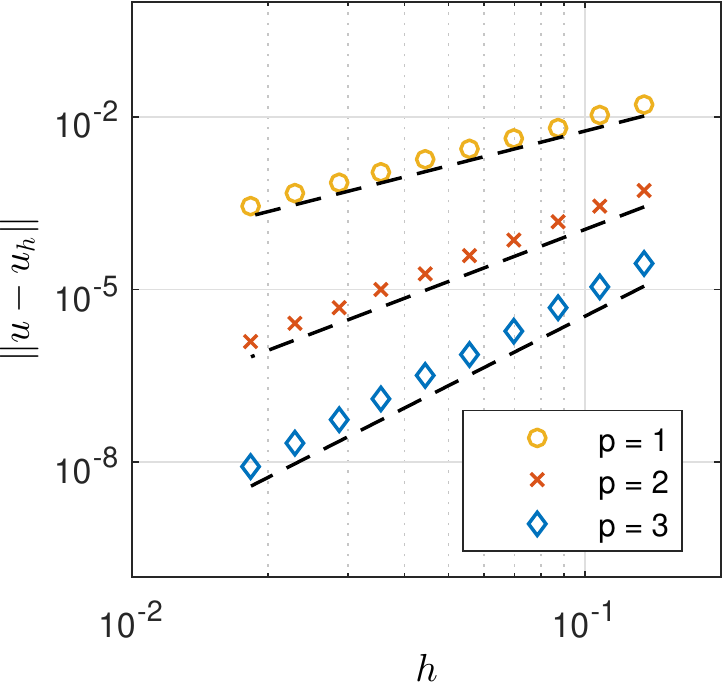}
\end{subfigure}
\begin{subfigure}[t]{0.48\textwidth}\centering
\includegraphics[width=0.9\textwidth]{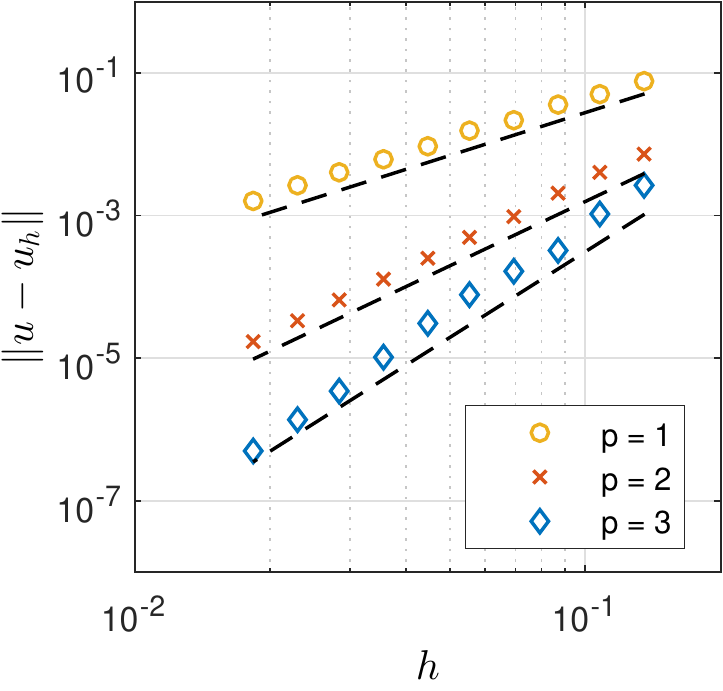}
\end{subfigure}
\caption{Convergence results in $L^2$ norm for the two model problems. The dashed reference lines are $h^{p+1}$.}
\label{fig:conv-L2}
\end{figure}

\subsection{Stability} \label{sec:stability}
\paragraph{Patch Position in the Background Mesh.}
Depending on how a reference patch $\hatOmega_i$ is positioned in $[0,1]^2$ the intersection with the background mesh may produce situations with arbitrary small cut elements. To demonstrate the stability of the method with regard to different cut situations we produce statistical data by randomly rotating each reference patch $\hatOmega_i$ in the background mesh to give random cut situations and repeating the simulation N times.  The standard deviation of the energy and $L^2$ errors in these simulations are presented in Figure~\ref{fig:rotStudy}.

\begin{figure}
\centering
\begin{subfigure}[t]{0.48\textwidth}\centering
\includegraphics[width=0.9\textwidth]{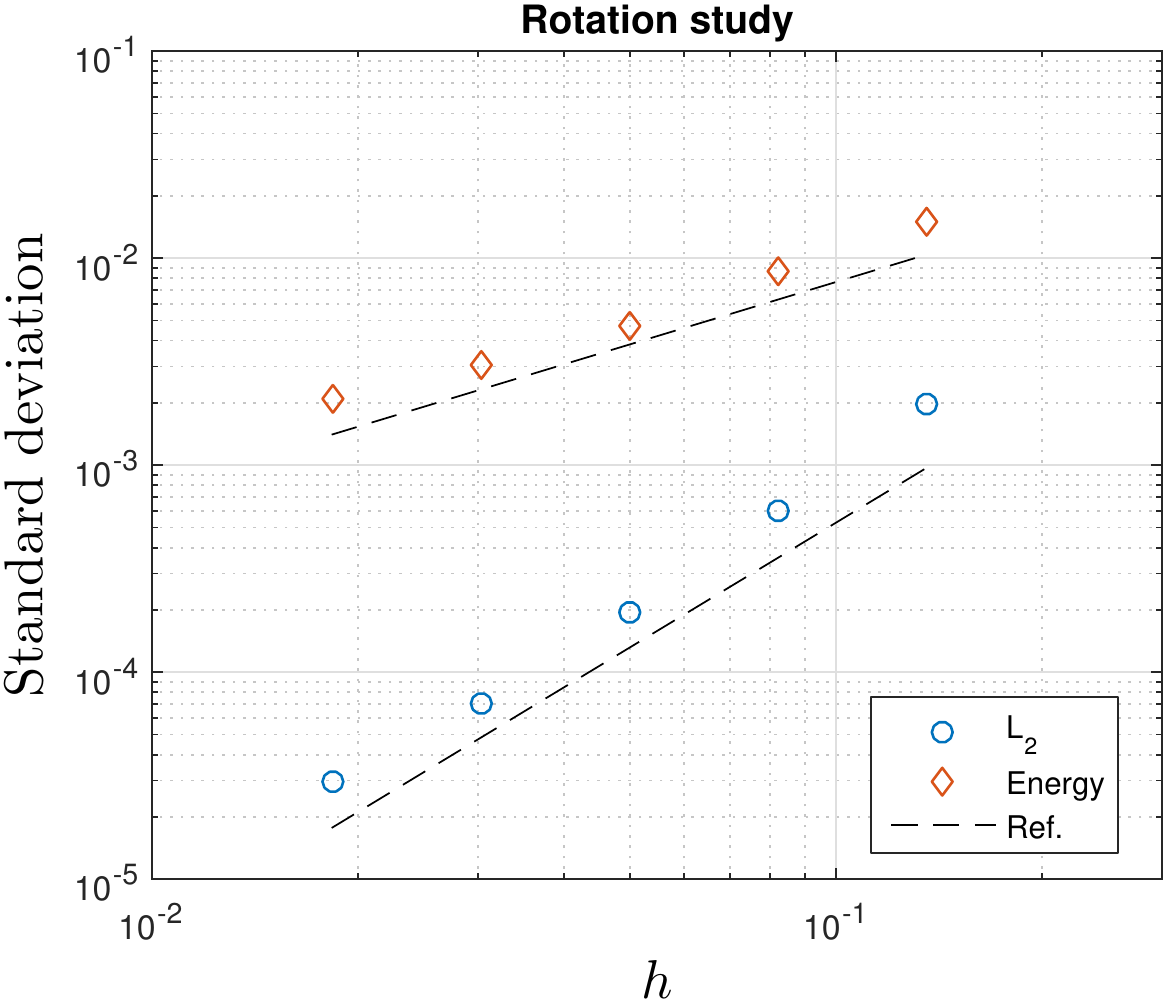}
\end{subfigure}
\begin{subfigure}[t]{0.48\textwidth}\centering
\includegraphics[width=0.9\textwidth]{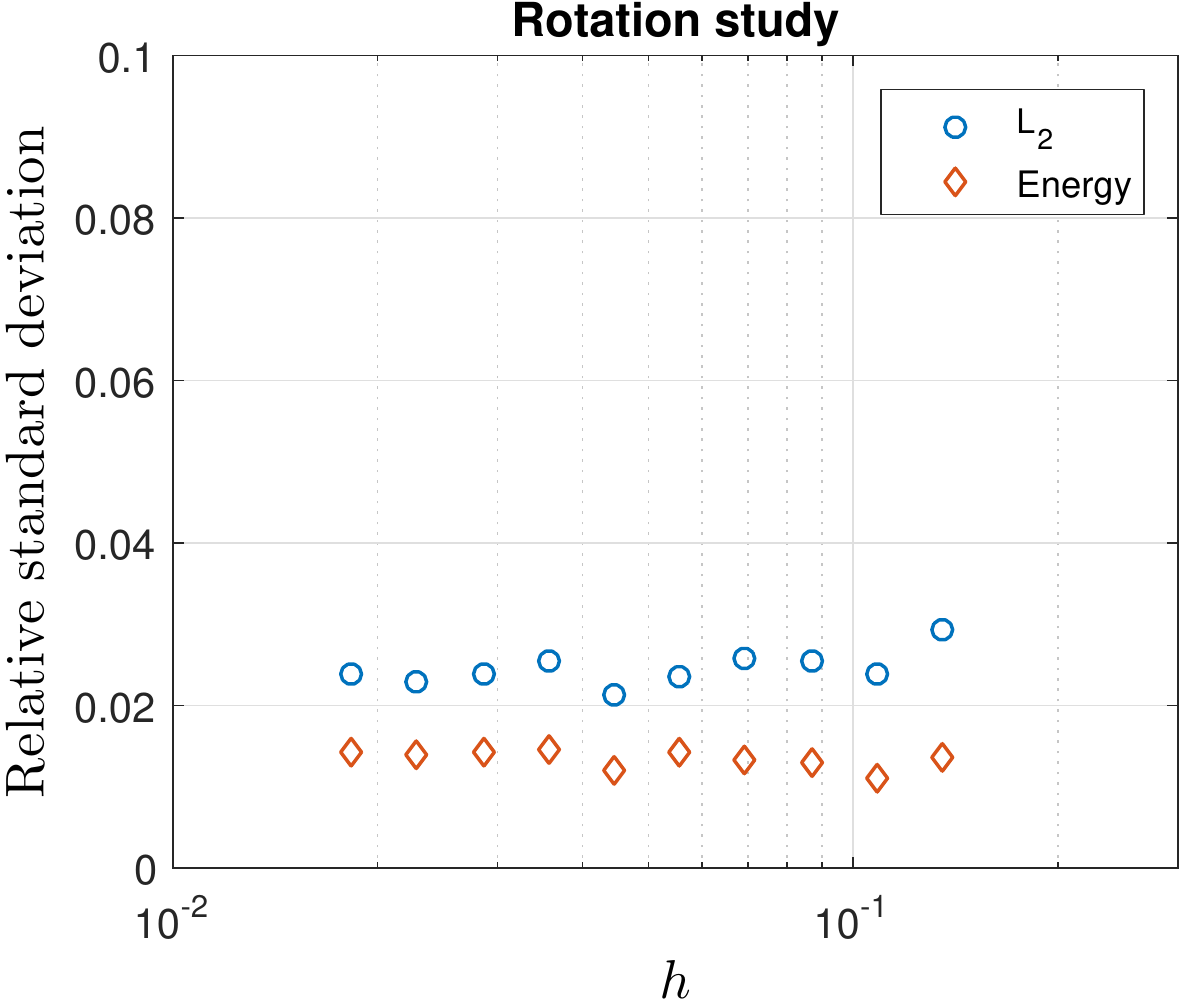}
\label{fig:rotRel}
    \end{subfigure}
\caption{
Standard deviation (left) and relative standard deviation (right) of the error in the energy norm and $L^2$ norm against the mesh size $h$ for the sphere model problem with random placement of the reference patches in the background mesh. Bilinear finite elements are used ($p=1$) and the
reference lines in the left figure are $\mathcal{O}(h)$ and $\mathcal{O}(h^{2})$.}
      \label{fig:rotStudy}
\end{figure}

\paragraph{Condition Number.}
The discrete problem can become arbitrarily ill conditioned if the stabilization term $j_{h,i}$ is not included in the form $A_h$. To capture this instability we estimate the condition number of the stiffness matrix for numerous patch positions in the reference domain, producing different cut situations. This is illustrated in Figure~\ref{fig:kappa_h} where we estimate the condition number for both the stabilized and unstabilized system in random cut situations. Note that the condition number for the stabilized stiffness matrix scales as $\mathcal{O}(h^{-2})$ in agreement with the bound proven in Theorem~\ref{thm:condition-number}.

\begin{figure}
\centering
\begin{subfigure}[t]{0.48\textwidth}\centering
\includegraphics[width=0.95\textwidth]{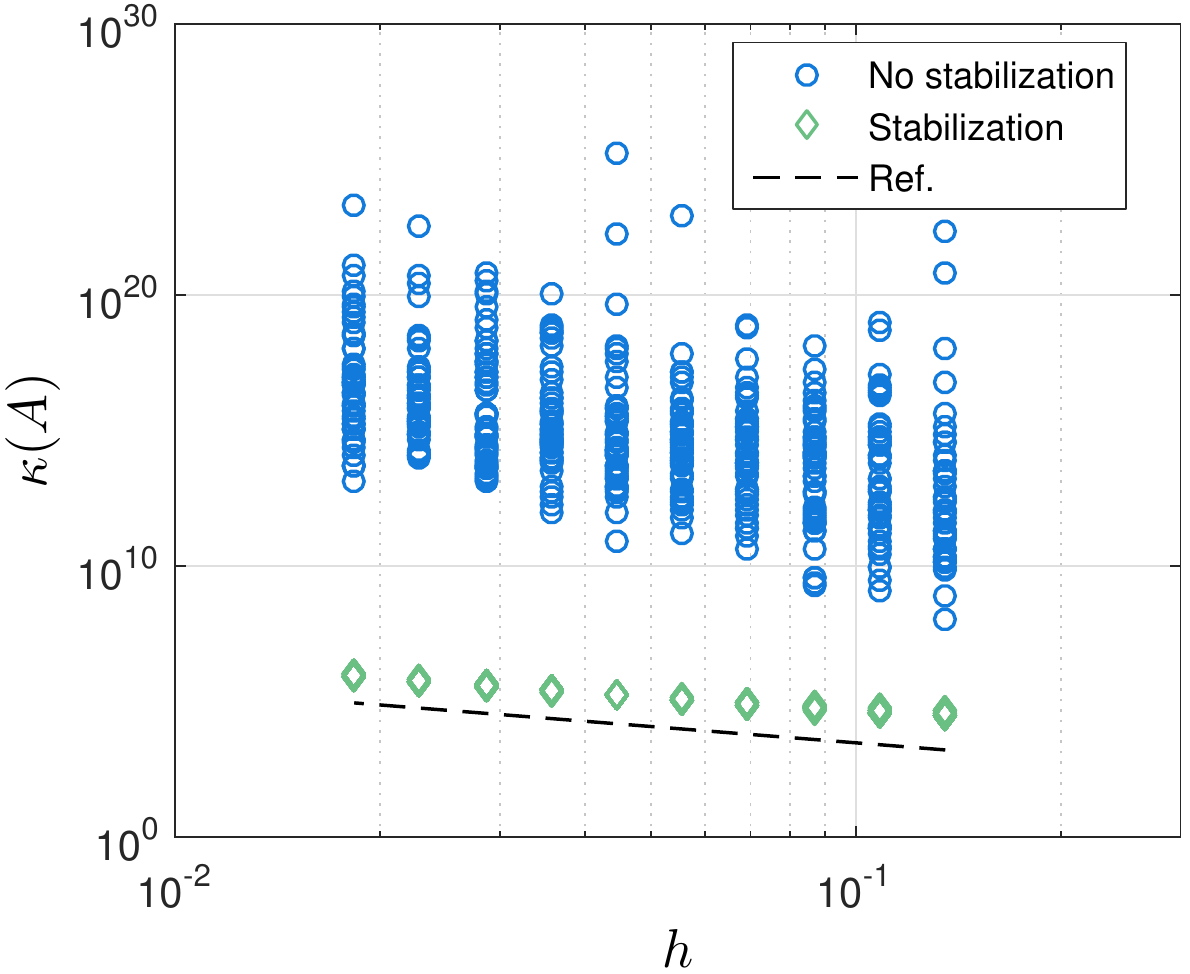}
\caption{}
\label{fig:kappa_h}
\end{subfigure}
\begin{subfigure}[t]{0.48\textwidth}\centering
\includegraphics[width=0.95\textwidth]{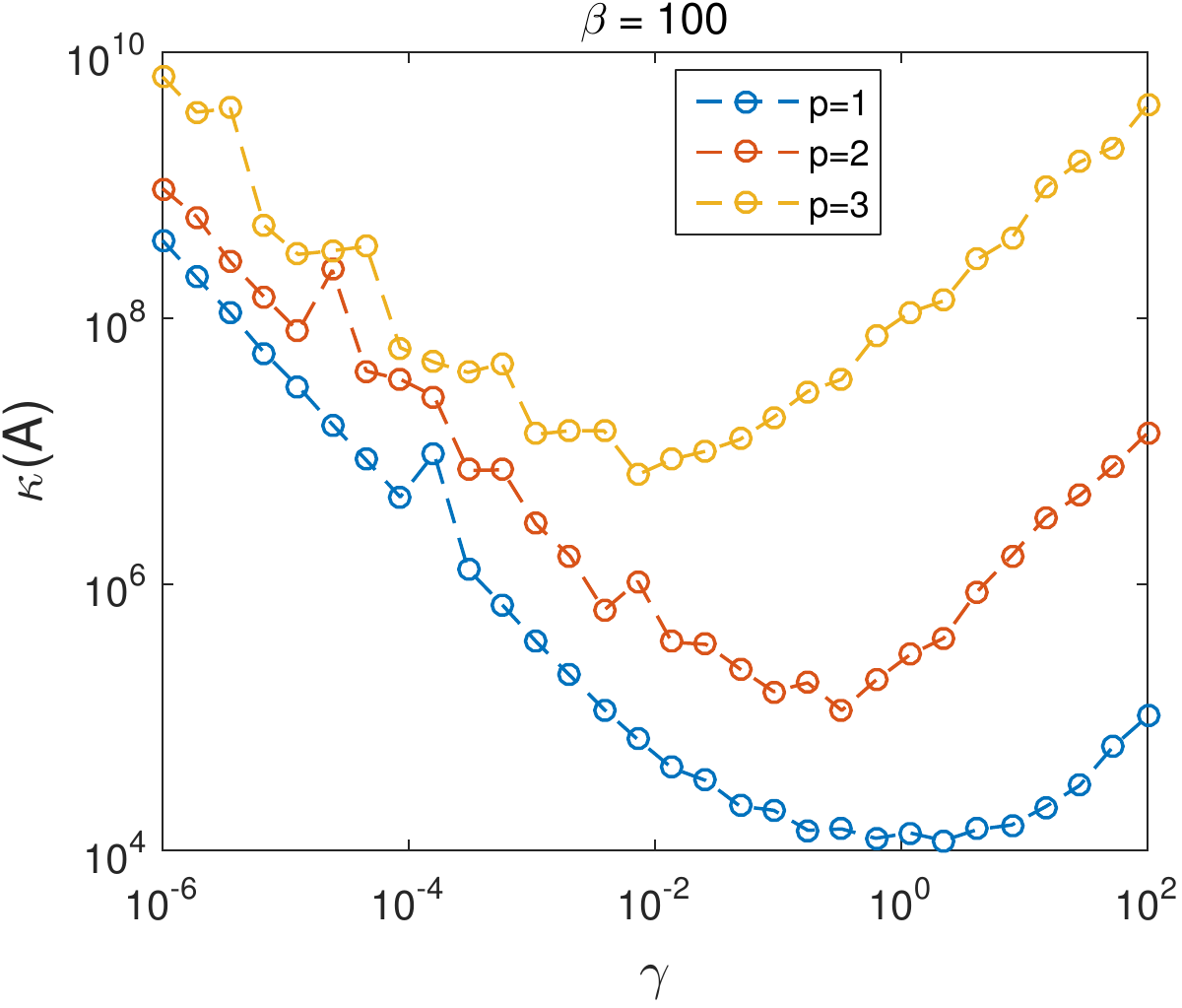}
\caption{}
\label{fig:kappaStab}
\end{subfigure}
\caption{Numerical studies of the stiffness matrix condition number in the unit sphere model problem.
(a) The condition number $\kappa(\mathbf{A})$ against the mesh size $h$ for 50 random positions of the reference patches in the background grids ($p=1$). The reference line is $\mathcal{O}(h^{-2})$.
(b) The condition number $\kappa(\mathbf{A})$ as a function of the stability parameter $\stab$.
}
\label{fig:condest}
\end{figure}

\paragraph{Choice of Stability Parameter $\boldsymbol\gamma$.}

For a fixed mesh size $h$ we investigate how the size of the stabilization parameters $\gamma_k$, $k=1,\dots,p$, affect numerical stability, i.e. the condition number $\kappa$, and the size of the error in the solution. Assuming all stabilization parameters take on the same value, i.e. $\gamma_k=\gamma$, we present a numerical study of this in Figure~\ref{fig:kappaStab} respectively in Figure~\ref{fig:stabErr}.
For small values of $\gamma$ we increasing condition numbers resulting in numerical instabilities, see Figure~\ref{fig:kappaStab}, and for large values of $\gamma$ we note that the stabilization term $j_{h,i}$ will start to impact the solution leading to larger errors, see Figure~\ref{fig:stabErr}. A good middle ground seems to be $\gamma=10^{-2}$.

\begin{figure}
\centering
\begin{subfigure}[t]{0.48\textwidth}\centering
\includegraphics[width=0.95\textwidth]{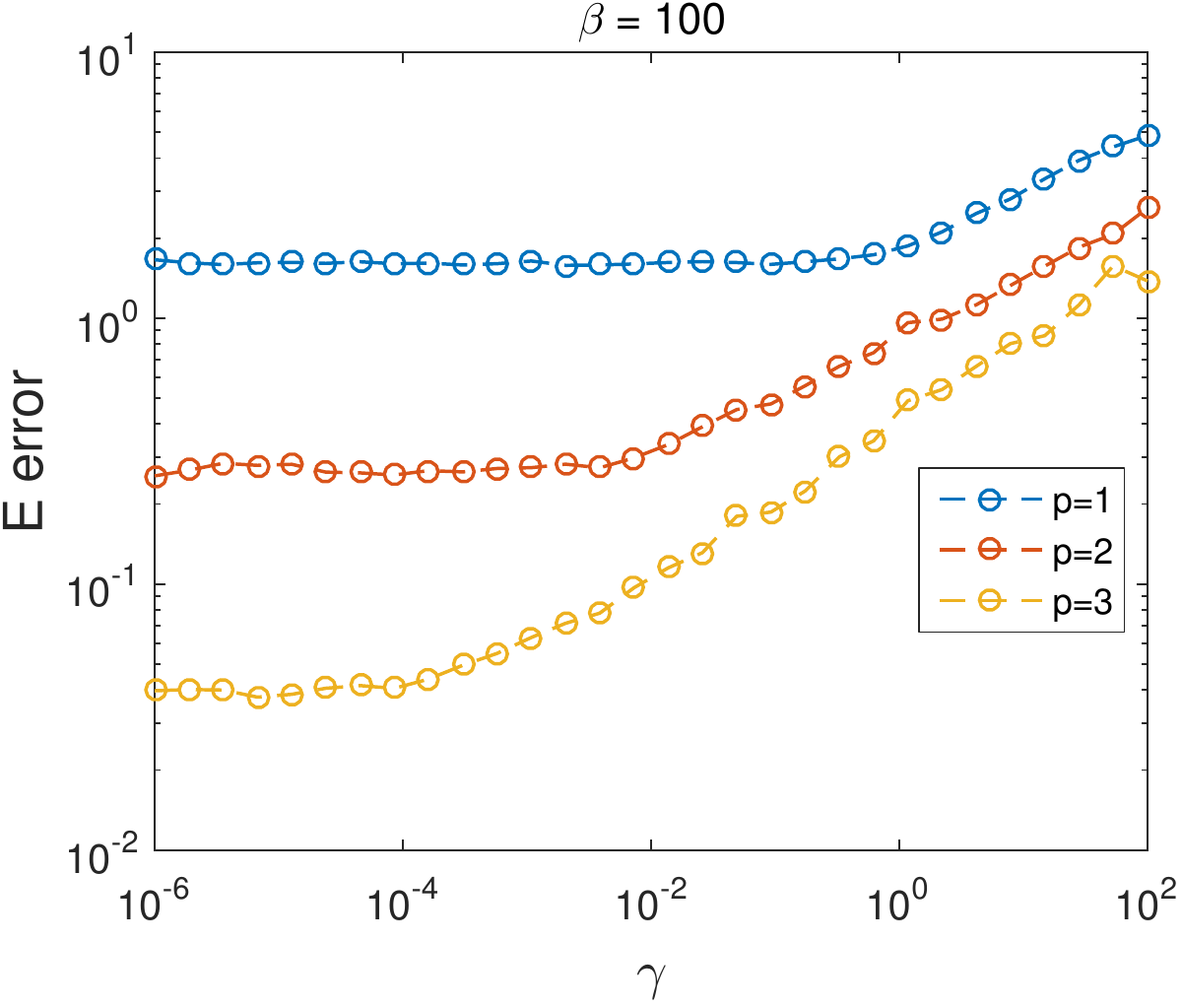}
\end{subfigure}
\begin{subfigure}[t]{0.48\textwidth}\centering
\includegraphics[width=0.95\textwidth]{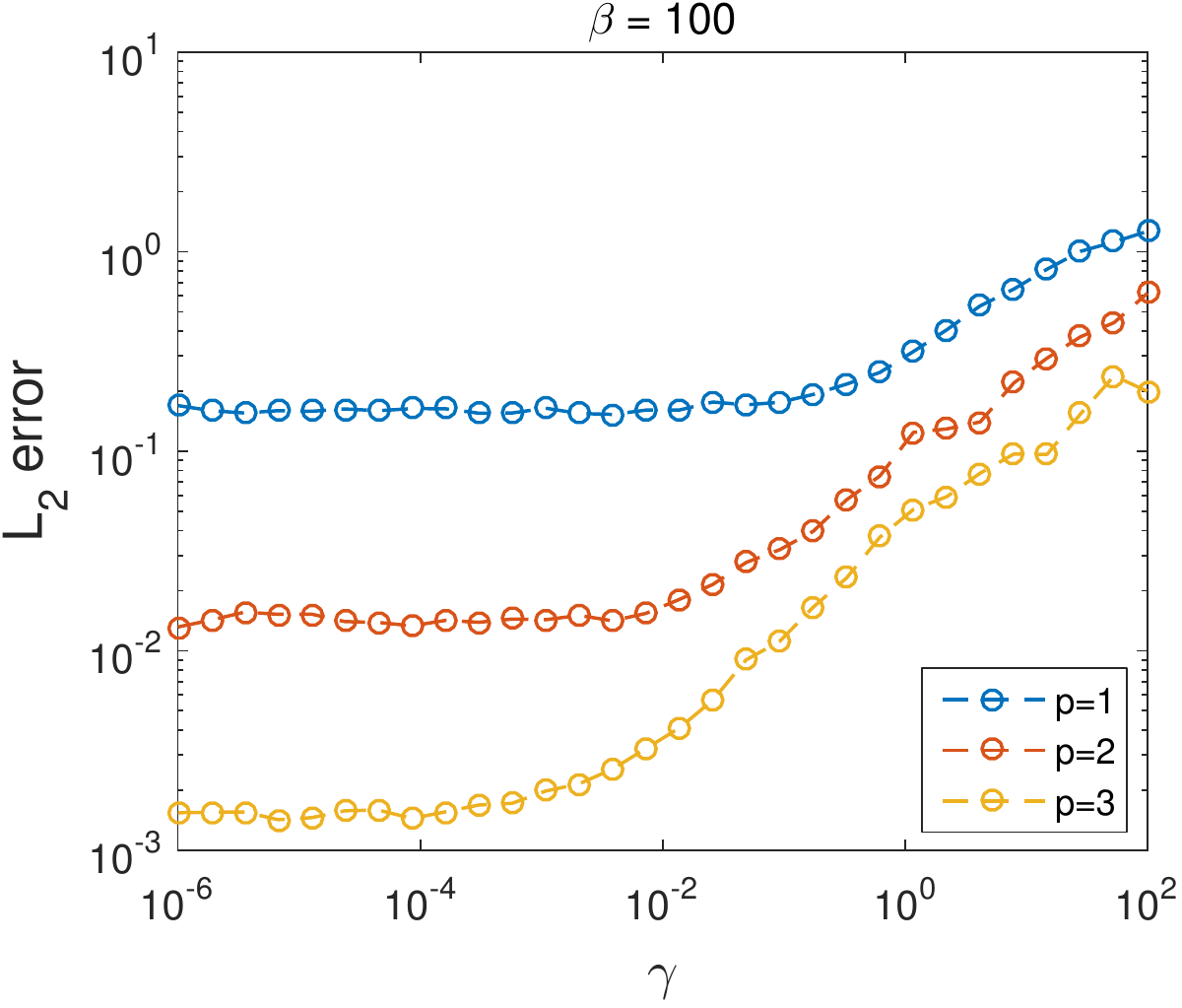}
\end{subfigure}
\caption{Error in the energy norm and in $L_2$ norm as a function of the stability parameter $\stab$. }
\label{fig:stabErr}
\end{figure}

\subsection{Numerical Examples}
\paragraph{Surface with Boundary.}
With the simple adaption of the method to boundary conditions described in Section~\ref{sec:method} we give results of a problem with boundary in Figure~\ref{fig:torusBoundary}, where we have both non-homogeneous Dirichlet conditions and Neumann conditions.

\paragraph{Klein Bottle.}
The Klein bottle is a closed non-orientable surface for which it exists no embedding in $\IR^3$. Let $\Omega$ in Cartesian coordinates be described by the parametrization
\begin{align}
    x(\theta,\phi) &= -\frac{2}{15} \cos \theta (3 \cos{\phi}-30 \sin{\theta}+90 \cos^4{\theta} \sin{\theta} 
\\ 
\nonumber &\quad -60 \cos^6{\theta} \sin{\theta}+5 \cos{\theta} \cos{\phi} \sin{\theta}) 
\\
y(\theta,\phi) &= -\frac{1}{15} \sin \theta (3 \cos{\phi}-3 \cos^2{\theta} \cos{\phi}-48 \cos^4{\theta} \cos{\phi}+ 48 \cos^6{\theta} \cos{\phi}
\\
\nonumber &\quad -60 \sin{\theta}+5 \cos{\theta} \cos{\phi} \sin{\theta}-5 \cos^3{\theta} \cos{\phi} \sin{\theta}-80 \cos^5{\theta} \cos{\phi} \sin{\theta}
\\
\nonumber  &\quad +80 \cos^7{\theta} \cos{\phi} \sin{\theta}) 
\\
    z(\theta,\phi) &= \frac{2}{15} (3+5 \cos{\theta} \sin{\theta}) \sin{\phi}
\end{align}
for $0 \leq \theta < \pi$ and $0 \leq \phi < 2 \pi$. We manufacture a problem with the analytical solution $u = 3\cos^2 \theta \sin \phi - \sin^3 \phi$ and the resulting
finite element solution is presented in Figure~\ref{fig:klein}.

\begin{figure}
\centering
\begin{subfigure}[b]{0.43\textwidth}\centering
\includegraphics[width=0.9\textwidth]{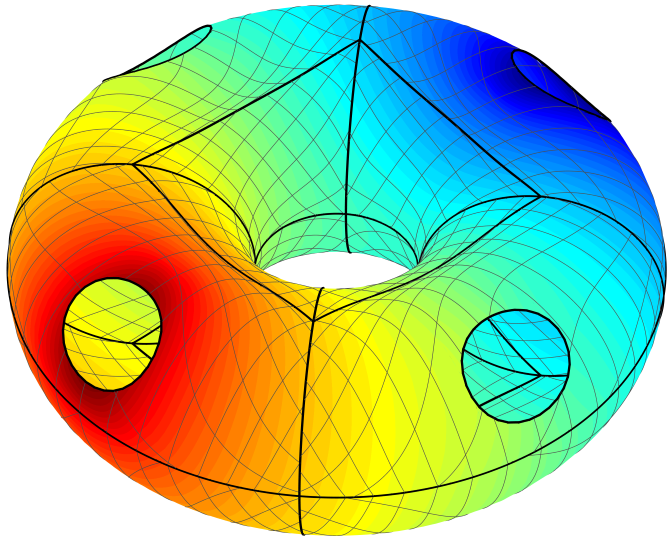}
\caption{Surface with boundary}  \label{fig:torusBoundary}
\end{subfigure}
\begin{subfigure}[b]{0.55\textwidth}\centering
\includegraphics[width=0.9\textwidth]{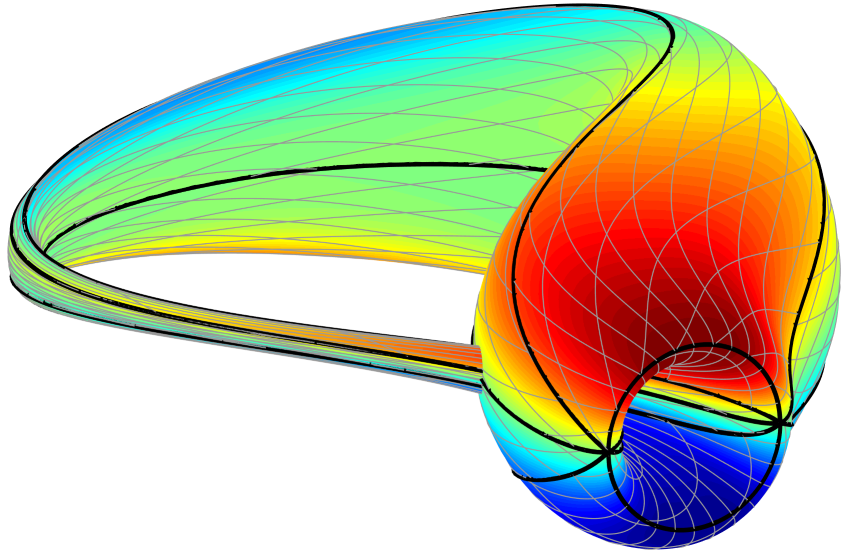}
\caption{Klein bottle} \label{fig:klein}
\end{subfigure}
\caption{(a) Solution to a problem with Dirichlet boundary conditions on the top right, $u=0$, and bottom left, $u=10$, holes and homogeneous Neumann boundary conditions on the top left and bottom right holes.
(b) Solution to a problem posed on a Klein bottle; a non-orientable surface for which there exist no embedding in $\IR^3$.}
\end{figure}

\paragraph{Surface with Sharp Interfaces.}
Let $\Omega$ be the closed surface to the half solid torus defined via \eqref{eq:torusPar} and $\theta \in [0,2 \pi],\,\phi \in [\frac{\pi}{2},\frac{3 \pi}{2}]$ and $r \leq 0.6$. The resulting geometry consists of half a torus and two circular discs. We manufacture a problem by choosing the same analytical solution as for the torus model problem on the torus part and on the discs we make the ansatz of a single cubic Hermite polynomial in the radial direction with zero solution and radial derivative in the disc center. The analytical solution on the discs are then derived from the interface conditions. In Figure~\ref{fig:sharp-interface} the solution and gradient magnitude of the finite element solution are displayed, and both flow nicely over the interfaces.

\begin{figure}
\centering
\begin{subfigure}[b]{0.47\textwidth}\centering
\includegraphics[width=0.9\textwidth]{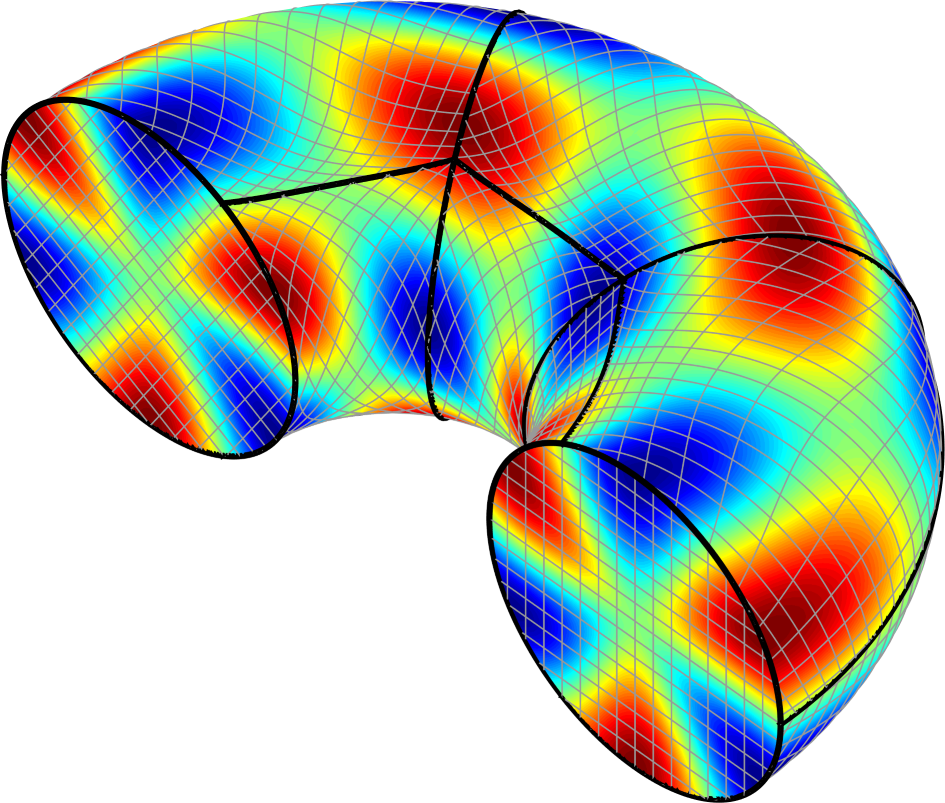}
\caption{Solution}  \label{edgeSol}
\end{subfigure}
\begin{subfigure}[b]{0.47\textwidth}\centering
\includegraphics[width=0.9\textwidth]{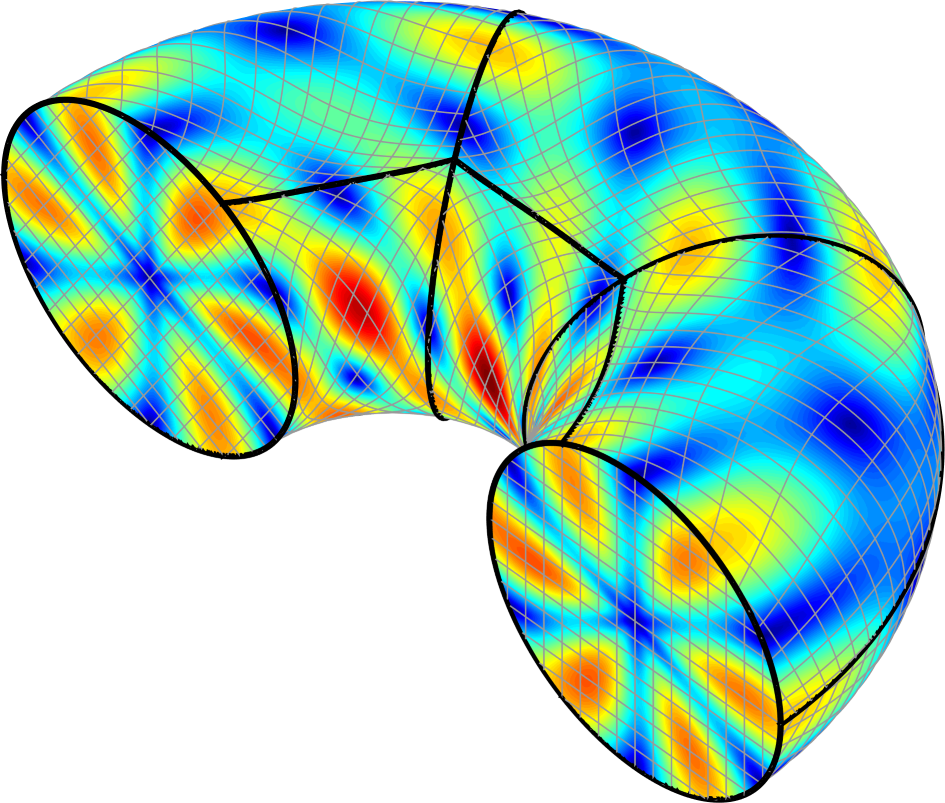}
\caption{Magnitude of gradient} \label{edgeGrad}
\end{subfigure}
\caption{Finite element solution to a problem with a sharp interface. Note that both the solution (a) and the gradient magnitude (b) flows nicely over the interface.}
\label{fig:sharp-interface}
\end{figure} 

\section{Summary and Future Work} \label{sec:conclusions}

We have presented and analysed a higher order cut finite element method for elliptic problems on multipatch surfaces. The method has the following fundamental features:
\begin{itemize}
\item Patches are described by mappings from a reference domain and trim curves.
\item On each patch a mesh is constructed using structured grids in the reference domain.
\item The discrete solution is coupled between the patchwise meshes by enforcing  interface conditions using Nitsche's method.
\item On each patch we handle elements cut by trim curves by adding certain stabilization terms.
\item The stability and error analysis is independent of how the trim curves cut the mesh.
\end{itemize}

\paragraph{Real Applications.}
While we in this work consider the Laplace--Beltrami operator as a model problem, there are many real problems posed on surfaces to which the same framework for dealing with multipatch surfaces effectively could be applied.
For example, there is a great interest in structural mechanics for modeling membranes, plates and shells, and modeling of thin films and lubrication also occur on surfaces.

\paragraph{Extended Analysis.}
In the analysis we assume that, at the interface, the trim curves on both patches map exactly onto the same interface curve. However, this is typically not the case when working with geometries extracted from CAD due to the discrete representation of the trim curves. Therefore a useful extension of the analysis would be to consider gaps in the geometry.
Another useful extension would be higher order PDE which are common for problems on surfaces and in this setting we can easily construct a tensor product basis with higher order continuity properties, i.e. where the restriction of the finite element space to each patch is a subspace to the proper Hilbert space.

\paragraph{Isogeometric Analysis.}
As the presented multipatch method is based on an exact description of the geometry by parametric mappings and features higher order elements it fits perfectly into the framework of isogeometric analysis \cite{IGA,IGABook}. The CutFEM approach also allows for convenient construction of structured meshes equipped with tensor product spline basis functions.

\appendix
\section{Differential Operators on Surfaces} \label{appendix:operators}
We provide details for the local forms of the divergence and Laplace--Beltrami 
operator as well as a derivation of Green's formula on a surface using only calculus in the reference coordinates and some basic linear algebra.
\paragraph{Divergence.} Starting from the definition 
\begin{equation}
-(\divv v, w)_{\omega} = (v,\nabla w)_{\omega}
\end{equation}
for $v,w \in C^\infty_0(\omega)$, for $\omega \subset \Omega_\ia$, 
we have the identities
\begin{align}
-(\widehat{\divv v}, \hatw)_{\hatomega} &= (\hatv,\widehat{\nabla w})_{\hatomega} 
\\
&= \int_{\hatomega} g(\hatv,\widehat{ \nabla w}) |\hatG|^{1/2} d\hatx 
\\
&= \int_{\hatomega} \hatv \cdot \hatnabla \hatw |\hatG|^{1/2} d\hatx 
\\
&= - \int_{\hatomega}  \hatnabla \cdot ( |\hatG|^{1/2} \hatv)  \hatw d\hatx
\\
&=- \int_{\hatomega}  |\hatG|^{-1/2}(\hatnabla \cdot |\hatG|^{1/2} \hatv)  
\hatw |\hatG|^{1/2} d\hatx
\\
&= ( |\hatG|^{-1/2}(\hatnabla \cdot |\hatG|^{1/2} \hatv) ,  \hatw)_{\hatomega}
\end{align}
where we integrated using standard Green's formula in local coordinates. 
Thus we conclude that 
\begin{equation}
\widehat{\divv v} = |\hatG|^{-1/2}(\hatnabla \cdot |\hatG|^{1/2} \hatv) 
\end{equation}

\paragraph{The Laplace--Beltrami Operator.} Using the definition 
of the Laplace--Beltrami operator (\ref{def:Laplace-Beltrami}) we 
conclude that in local coordinates
\begin{equation}\label{def:divergence-local}
\widehat{\Delta v} 
= \widehat{\divv ( \nabla v)}
= |\hatG|^{-1/2} \hatnabla \cdot (|\hatG|^{1/2} \widehat {\nabla v} )
= |\hatG|^{-1/2} \hatnabla \cdot (|\hatG|^{1/2} \hatG^{-1} \hatnabla \hatv )
\end{equation}
where we used (\ref{eq:gradient-local}).
\paragraph{Green's Formula.} We have the identities
\begin{align}
-(\Delta v, w)_\omega 
&=
-\int_{\hatomega} \widehat{\Delta v}\hatw |\hatG|^{1/2} d \hatx
\\
&= 
-\int_{\hatomega} |\hatG|^{-1/2} \hatnabla \cdot(|\hatG|^{1/2} \hatG^{-1} 
\hatnabla \hatv) \hatw |\hatG|^{1/2} d \hatx
\\
&= 
-\int_{\hatomega}  \hatnabla \cdot(|\hatG|^{1/2} \widehat{\nabla v}) \hatw  d \hatx
\\
&= 
\int_{\hatomega}  |\hatG|^{1/2} \widehat{\nabla v} \cdot \hatnabla\hatw  d \hatx
- \int_{\partial \hatomega}  |\hatG|^{1/2} \hatnu \cdot \widehat{\nabla v} \hatw  
d \hatgamma
\\ \label{eq:appendix-b}
&= 
\underbrace{\int_{\hatomega}   g(\widehat{\nabla v} ,\widehat{\nabla w})  |\hatG|^{1/2} d \hatx}_{\int_\omega \nabla v \cdot \nabla w \, dx}
- \underbrace{\int_{\partial \hatomega}  g(\hatG^{-1}\hatnu, \widehat{\nabla v}) \hatw  
|\hatG|^{1/2} d \hatgamma }_{\bigstar=\int_{\partial \omega}  n\cdot \nabla v  \, w \,   d \gamma}
\end{align}
where we will now verify the identity  $\bigstar =  \int_{\partial \omega}  n\cdot \nabla v  \, w \,   d \gamma$. We note that $\hatnu$ is the unit exterior normal to $\partial \omega$, with respect to the 
usual $\IR^2$ inner product in $T(\hatOmega_\ia)$, and that 
\begin{equation}
\hatn = \frac{\hatG^{-1} \hatnu}{\| \hatG^{-1} \hatnu \|_g} 
\end{equation}
see (\ref{eq:normal-coordinates}).
Thus the term $\bigstar$ on the right hand side in (\ref{eq:appendix-b}) may 
be written in the form
\begin{align}
\bigstar = \int_{\partial \hatomega}  g(\hatG^{-1}\hatnu, \widehat{\nabla v}) \hatw  
|\hatG|^{1/2} d \hatgamma
&=
\int_{\partial \hatomega}  g(\hatn, \widehat{\nabla v}) \hatw  
\| \hatG^{-1} \hatnu \|_g |\hatG|^{1/2} d \hatgamma
\end{align}
and we will next study the measure $\| \hatG^{-1} \hatnu \|_g |\hatG|^{1/2} d \hatgamma$ in more detail.

\paragraph{The Two-Dimensional Case.} Using the identity 
\begin{equation}
\hatG^{-1} = |\hatG|^{-1} S^T \hatG S 
\end{equation}
where 
\begin{equation}
S =\left( 
\begin{matrix}
0 & -1
\\
1 & 0
\end{matrix}
\right)
\end{equation}
we have $\hattau = S \hatnu$, where $\hattau$ is the unit with respect to the Euclidean inner product tangent vector to the curve $\partial \hatomega$,  and
\begin{align}
|\hatG| \| \hatG^{-1} \hatnu \|_g^2 &= |\hatG| \hatnu \cdot \hatG^{-1} \cdot \hatnu
\\ 
&= \hatnu \cdot (S^T \hatG S) \cdot \hatnu
\\
&= (S\hatnu) \cdot \hatG \cdot (S \hatnu)
\\
&= \hattau \cdot \hatG \cdot \hattau
\\
&=\| \hattau \|^2_g
\end{align}
and therefore we recover the standard curve measure. Thus we conclude 
that 
\begin{align}
\int_{\partial \hatomega}  g(\hatG^{-1}\hatnu, \widehat{\nabla v}) \hatw  
|\hatG|^{1/2} d \hatgamma
&=
\int_{\partial \hatomega}  g(\hatn, \widehat{\nabla v}) \hatw  
\| \hatG^{-1} \hatnu \|_g |\hatG|^{1/2} d \hatgamma
\\
&= \int_{\partial \hatomega}  g(\hatn, \widehat{\nabla v}) \hatw  
\|\hattau \|_g d \hatgamma
\\
&= \int_{\partial \omega}  n\cdot \nabla v  \, w \,   d \gamma
\end{align}
which together with (\ref{eq:appendix-b}) concludes the derivation of Green's formula on the surface $\omega$.

\paragraph{The General Case.} A more general approach which also holds 
in higher dimension is to use the identity 
\begin{equation}
|A + a \otimes b| = |A| + |A| a\cdot A^{-1} \cdot b
\end{equation}
where $A$ is a square $n \times n$ matrix and $a$ and $b$ are $n$ vectors. 
We then obtain
\begin{align}
|\hatG| \| \hatG^{-1} \hatnu \|_g^2 &= |\hatG| \hatnu \cdot \hatG^{-1} \cdot \hatnu
\\
&=|\hatG + \hatnu \otimes \hatnu | - |\hatG|
\end{align}
Now we may chose an orthonormal basis in $\IR^n$ consisting of 
$\hatnu$ and $n-1$ tangent vectors $\{t_i\}_{i=2}^{n}$. Let $P=I-\hatnu \otimes \hatnu$  be the projection onto the tangent plane and then we 
have the identity
\begin{align}
|\hatG + \hatnu \otimes \hatnu |
&=|\hatG| +  |P \hatG P|
\end{align}
where $|P \hatG P|$ is the $n-1$ determinant of the tangent part 
$P\hatG P$ of $\hatG$.
This follows directly from the fact that in normal-tangent 
coordinates $\hatG + \hatnu \otimes \hatnu$ takes the form
\begin{equation}
\underbrace{
\left|
\begin{matrix}
\tildeg_{11} + 1 & \tildeg_{12} &\dots &\tildeg_{1n}
\\
\tildeg_{21} & \tildeg_{22} & \dots & \tildeg_{2n}
\\
\vdots          & \vdots        & \vdots & \vdots
\\
\tildeg_{n1} & \tildeg_{n2} & \dots & \tildeg_{nn}
\end{matrix}
\right|}_{|G + \hatnu \otimes \hatnu|}
=
\underbrace{\left|
\begin{matrix}
\tildeg_{11}  & \tildeg_{12} &\dots &\tildeg_{1n}
\\
\tildeg_{21} & \tildeg_{22} & \dots & \tildeg_{2n}
\\
\vdots          & \vdots        & \vdots & \vdots
\\
\tildeg_{n1} & \tildeg_{n2} & \dots & \tildeg_{nn}
\end{matrix}
\right|}_{|G|}
+
\underbrace{
\left|
\begin{matrix}
 \tildeg_{22} & \dots & \tildeg_{2n}
\\
\vdots        & \vdots & \vdots
\\
 \tildeg_{n2} & \dots & \tildeg_{nn}
\end{matrix}
\right|
}_{|PGP|}
\end{equation}
with $\widetilde{G}$ the matrix representation of $\hatG$ in a tangent-normal coordinate system. We may thus conclude that 
\begin{align}
|\hatG|^{1/2} \| \hatG^{-1} \hatnu \|_g &=|P \hatG P|^{1/2}
\end{align}
which is the appropriate measure on $\partial \omega$. Note also that in the two dimensional case $P \hatG P $ has rank one and the determinant equals the absolute value of the scalar $\hattau \cdot \hatG \cdot \hattau$ and thus 
\begin{equation}
| P \hatG P | = |\hattau \cdot \hatG \cdot \hattau | = \| \hattau \|^2_g
\end{equation} 
which is consistent with the definition of the measure based on arclength measure.

\bibliographystyle{abbrv}
\bibliography{referencesCutFEM}

\end{document}